\documentclass[amscd,amssymb,verbatim,10pt]{amsart}

\title[A new discrete monotonicity formula]
{A new discrete monotonicity formula with application to a two-phase free boundary problem
in dimension two}

\usepackage{esint}
\usepackage{color}
\usepackage{hyperref}

\usepackage{caption}
\usepackage{subcaption}
\theoremstyle{plain}
\newtheorem{theorem}{Theorem}[section]
\newtheorem{lemma}[theorem]{Lemma}

\newtheorem{prop}[theorem]{Proposition}
\newtheorem{corollary}[theorem]{Corollary}
\newtheorem{remark}[theorem]{Remark}
\theoremstyle{definition}
\newtheorem{defn}[theorem]{Definition}

\renewcommand\epsilon\varepsilon % use curly epsilon
\renewcommand\phi\varphi % use curly phi
\newcommand\supp{\operatorname{supp}} %support
\newcommand\dist{\operatorname{dist}} % distance
\newcommand\R{\mathbb{R}} % real axis
\numberwithin{equation}{section}
\newcommand\e{{\epsilon}}
\newcommand\na{{\nabla}}
\newcommand\fb[1]{\partial\{{#1>0}\}}

\newcommand{\ca}{\mathrm{cap}}

\newcommand{\mpar}[1]{%
  \marginpar{\colorbox{lime}{\parbox{\marginparwidth}{%
  \setstretch{0.5}\textcolor{black}{\scriptsize{#1}}}}}}
\usepackage[english]{babel}
\usepackage{blindtext}
\usepackage[svgnames]{xcolor}
\usepackage[doublespacing]{setspace}
\usepackage{amsmath, amssymb}

\usepackage{comment}
\usepackage{graphicx, geometry}
\usepackage[a4,center]{crop}

\author[S. Dipierro]{Serena Dipierro}
\address[Serena Dipierro]{Maxwell Institute for Mathematical Sciences and
School of Mathematics, University of Edinburgh,
James Clerk Maxwell Building, Peter Guthrie Tait Road,
Edinburgh EH9 3FD,
United Kingdom.
}
\email{serena.dipierro@ed.ac.uk}

\author[A.L. Karakhanyan]{Aram L. Karakhanyan}
\address[Aram Karakhanyan]{Maxwell Institute for Mathematical Sciences and
School of Mathematics, University of Edinburgh,
James Clerk Maxwell Building, Peter Guthrie Tait Road,
Edinburgh EH9 3FD,
United Kingdom.
}
\email{aram.karakhanyan@ed.ac.uk}

\thanks{2000 Mathematics Subject Classification. Primary 35R35, 35J92. 
Keywords: Free boundary regularity, two phase, $p-$Laplace, partial regularity. 
\\
The authors were partially supported by EPSRC grant EP/K024566/1.
 }
\newcommand\rred[1]{{\color{red} #1}}

\begin{document}
\maketitle

\begin{abstract}
We continue the analysis of the two-phase free boundary problems 
initiated in  \cite{DK}, where we studied the linear growth of minimizers in
a Bernoulli type free boundary problem at the non-flat points
and the related regularity of free boundary. There, we also defined the functional
$$\phi_p(r,u,x_0)=\frac1{r^4}\int_{B_r(x_0)}\frac{|\na u^+(x)|^p}{|x-x_0|^{N-2}}dx\int_{B_r(x_0)}\frac{|\na u^-(x)|^p}{|x-x_0|^{N-2}}dx$$
where $x_0$ is a free boundary point, i.e. $x_0\in\partial\{u>0\}$
and~$u$ is a minimizer of the functional
$$J(u):=\int_{\Omega}|\nabla u|^p +\lambda_+^p\,\chi_{\{u>0\}} +\lambda_-^p\,\chi_{\{u\le 0\}},
$$
for some bounded smooth domain $\Omega\subset\R^N$ and positive constants 
$\lambda_\pm$ with $\Lambda:=\lambda_+^p-\lambda^p_->0$.

Here we 
show the discrete monotonicity of $\phi_p(r,u,x_0)$ in two
spatial dimensions at non-flat points, when $p$ is sufficiently close to 2, 
and then establish the linear growth. A new feature of our approach
is the anisotropic scaling argument discussed in Section~\ref{s1}.

\end{abstract}

\setcounter{tocdepth}{1}
\tableofcontents
%%%%%%%%%%%%%%%%%%%%%%%%%%%
%				
%			 	SECTION													
%																						
%%%%%%%%%%%%%%%%%%%%%%%%%%%
\section{Introduction}
Let $\Omega\subset\R^2$ be a bounded planar domain such that any function in the Sobolev  space $W^{1, p}(\Omega)$
has well-defined trace 
and $p\in(2,2+\epsilon)$, for a small, fixed $\epsilon>0$.
Assume that $u$ is a local minimizer of
\begin{equation}\label{Ju}
J(u):=\int_{\Omega}|\nabla u|^p +\lambda_+^p\,\chi_{\{u>0\}} +\lambda_-^p\,\chi_{\{u\le 0\}},
\quad u-g\in W_0^{1, p}(\Omega),
\end{equation}
where $\lambda_+$ and $\lambda_-$ are positive constants such
that~$\lambda_+^p-\lambda_-^p>0$ and $g\in W^{1, p}(\Omega)$ is a prescribed boundary datum.
In what follows $\chi_U$ denotes the characteristic function of the set~$U\subset\R^2$.

The variational problem for \eqref{Ju} is called the Bernoulli-type free boundary problem
and it models a number of interesting phenomena, notably planar cavitational flow
of one or two perfect fluids (see \cite{BZ} Chapter 9.11), the equilibrium configuration for heat or
electrostatic energy optimization in higher dimensions, (e.g. heat flow with power Fourier  law) 
and the dynamics of non-Newtonian fluids 
when the velocity obeys the power law ${\bf v}=\nabla \psi |\nabla \psi|^{\frac1s-1}$
where $\psi$ is a stream function. Notice that $s=1$ corresponds to Newtonian fluids and 
$s$ is a physical parameter, see \cite{Astar}. 

\medskip 

For $p=2$ both the one phase and the two-phase problems have been extensively studied  
for variational \cite{ACF} as well as viscosity solutions \cite{Luis}. 
There is a significant difference between the one phase and two-phase
problems stipulated by a sign change of $u$ across the free boundary.
The main and only known method for proving the optimal regularity  
for the two-phase  problem is based on the monotonicity formula of 
Alt, Caffarelli and Friedman \cite{ACF} given by 
\begin{equation}
\phi(r,x_0)=\frac1{r^4}\int_{B_r(x_0)}\frac{|\na u^+(x)|^2}{|x-x_0|^{N-2}}dx
\int_{B_r(x_0)}\frac{|\na u^-(x)|^2}{|x-x_0|^{N-2}}dx,
\end{equation}
where $r>0$ and $x_0\in \partial\{u>0\}$. 
It is well-known that if  $u$ is a minimizer of \eqref{Ju} and $u^+:=\max\{0, u\}$ and 
$u^-:=-\min\{0, u\}$, then  
$\phi(r, x_0)$ is a non-decreasing function of $r$.
The monotonicity of $\phi$, combined with the coherent growth of 
$\frac1t\fint_{B_t(x_0)}u^+$ and $\frac1t\fint_{B_t(x_0)}u^-$, for $x_0\in \fb{u}$,  
gives uniform local upper linear bound for $u$, see \cite{ACF}.

\medskip

The key ingredient in the proof of the monotonicity formula in \cite{ACF} 
is the following geometric property of the 
eigenvalues of the Laplace-Beltrami operator on the unit sphere $\partial B_1$: 
let $\gamma_1, \gamma_2$  be the characteristic numbers 
corresponding to two complementary domains $\Gamma_1, \Gamma_2$ on $\partial B_1$, that is
$$\gamma_i(\gamma_i+N-2):=\inf_{v\in W^{1,2}_0(\Gamma_i)}\frac{\int_{\Gamma_i} 
|\nabla_\theta v|^2}{\int_{\Gamma_i} v^2},\quad i=1,2,$$
then  
\begin{equation}\label{char-nmbrs}
\gamma_1+\gamma_2\ge 2
\end{equation}
and the equality holds if and only if $\Gamma_1,\Gamma_2$ are two complementary hemispheres,  see \cite{Luis}, Chapter~12.

Note that in two spatial dimensions $\gamma_i$ is the square root of the eigenvalue corresponding to the
portion $\Gamma_i$ of the unit circle.  

In Section \ref{sec-eigen-1} we present some results related to the characteristc numbers 
and the eigenvalues of the $p$-Laplace-Beltrami operator for $p\neq2$.

\medskip

There are fewer results established for $p\neq 2$ for the two-phase problem. 
A partial result on the optimal regularity of $u$ is 
given in \cite{K1} under a smallness assumption on the Lebesgue density of 
the set $\{u\le 0\}$, and recently has been extended to a more general class of 
functionals in \cite{Moreira}. 

Our paper  contributes in the direction of optimal regularity and monotonicity formula techniques.
More precisely, we 
show that in two spatial dimensions $N=2$ (and for $p$ sufficiently close to 2)
the functional
\begin{equation}\label{varphi}
\varphi_p(r,u,x_0):=\frac{1}{r^4}\int_{B_r(x_0)}|\nabla u^+|^p\,\int_{B_r(x_0)}|\nabla u^-|^p
\end{equation}
is discrete monotone (see Theorem \ref{TH2} for the precise statement). 
Here~$r>0$ is small and~$x_0\in\Gamma:=\partial\{u>0\}$, 
being $\Gamma$ the free boundary.
Consequently, we prove that $\phi_p$ is bounded if the free boundary is not flat at $x_0$.

In fact, we establish  a dichotomy for~$\varphi_p$: either the free boundary is smooth at $x_0$ or
$\phi_p$ is discrete monotone. 

For this, we introduce
a suitable notion of flatness for the free boundary points characterizing the {\textit{flat}} points. 
It follows 
from the results of \cite{LN1, LN2} that at such points the free boundary must be  regular
provided that $u$ is also a viscosity solution in the sense of Definition \ref{def:visc}, see also the discussion 
in Section \ref{sec:visc}. The fact that minimizers of $J$ are viscosity solutions 
has been established in \cite{DK}. 

On the other hand, at {\textit{non-flat}} points, we prove that $\varphi_p$ is discrete monotone 
and we deduce from this the linear growth of $u$ near these points. 

\medskip

In the subsequent section, we present our main results.
A detailed plan about the organization of the
paper will then be presented at the end of Section~\ref{sec:main}.

%%%%%%%%%%%%%%%%%%%%%%%%%%%
%				
%			 	SECTION													
%																						
%%%%%%%%%%%%%%%%%%%%%%%%%%%
\section{Main Results}\label{sec:main}
In this section we formulate our main results. 
We will denote by $\Gamma:=\partial\{u>0\}$ the free boundary. 
Fix $x_0\in\Gamma$ and $h>0$, and 
consider the slab
\begin{equation}
S(h; x_0, \nu):=\{x\in \R^n : -h<(x-x_0)\cdot\nu<h\}
\end{equation}
where $\nu$ is a unit vector. Let $h_{\min}(x_0, r, \nu)$ be the
minimal height of the slab in the unit direction $\nu$ containing the free boundary in $B_r(x_0)$, i.e.
\begin{equation}\label{slab-flat}
h_{\min}(x_0, r, \nu):=\inf\{h : \partial\{u>0\}\cap B_r(x_0)\subset S(h; x_0, \nu)\cap B_r(x_0)\}.
\end{equation}
If we set
\begin{equation}\label{min-h}
h(x_0, r):=\inf_{\nu\in \mathbb S^n}h_{\min}(x_0, r, \nu)
\end{equation}
then $h(x_0, r)$ is non-decreasing in $r$.

Theorem to follow deals with the points where the free boundary is not
sufficiently flat.

\begin{theorem}\label{TH2}
Let $u$ be a local minimizer of the functional~$J$ defined in~\eqref{Ju}.
Then, there exist tame constants $p_0>2$, $r_0>0$ and $h_0>0$ such that if
\begin{eqnarray}
\label{p cond}&2<p<p_0\quad and \quad  r<r_0\\
\mbox{then the inequality}
\label{delta cond}
&h(x_0, r)\ge h_0r,\quad   {\mbox{ for }}x_0\in\Gamma\cap B_{3r},
\end{eqnarray}
implies that 
\begin{equation}\label{tripling}
\varphi_p(r,u,x_0)\le \varphi_p(3r,u,x_0), %\quad {\mbox { for any }}r\in(0, r_0),
\end{equation}
where $h(x, r)$ is defined by \eqref{min-h} and $\varphi_p$ by \eqref{varphi}.
\end{theorem}

Theorem \ref{TH2} says that if at the level $r$ the free boundary is not 
sufficiently flat then the $\phi_p$ energy at the level $r$ is controlled by 
the same energy at the tripled level $3r$. 

It is worthwhile to point out that in the proof of Theorem \ref{TH2}  
we use a compactness argument based on an anisotropic scaling in order 
to assure the non-degeneracy of an appropriately scaled function, 
thus avoiding the use of the knowledge of the linear growth from \cite{DK}. 

As a consequence, we have the following: 

\begin{theorem}\label{TH2-prime}
Let $u$ be a local minimizer of the functional~$J$ defined in~\eqref{Ju}, 
and let~$x_0\in\Gamma$ be a non-flat point of the free boundary, i.e. 
for any $r<r_0$ we have $h(x_0, r)\ge h_0r$, where $r_0$ is given by Theorem \ref{TH2}. 

Then,~$u$ has linear growth near~$x_0$.  
\end{theorem}

Observe that we always have that $u^+$ and $u^-$ have comparable rates of growth 
from the free boundary, thanks to Corollary \ref{cod-ball}, 
i.e. 
$$ \frac1r\fint_{B_r(x_0)}u^+\sim \frac1r\fint_{B_r(x_0)}u^-.$$ 
In order to conclude that each of these terms is bounded we apply Theorem \ref{TH2} 
to infer that 
the product $\frac1{r^2}\fint_{B_r(x_0)}u^+\fint_{B_r(x_0)}u^-$ is also bounded. 
This is where $\phi_p$ enters into the game and provides the 
necessary bound,  see Section \ref{sec-th2-0-pr}.

%%%
%%% ORGANIZATION OF PAPER
%%%

\subsection*{Outline}
In Section~\ref{sec:tec} we collect some basic
material that we will use throughout the paper. We also show a coherence result
(see Proposition~\ref{prop:tec}, P.4) by using a different strategy
with respect to the case $p=2$ (see \cite{ACF}), that we think has an independent interest.

Sections \ref{s1} and \ref{sec-th2-0-pr} are devoted to the proofs of Theorems \ref{TH2}
and \ref{TH2-prime}. 

In Section \ref{sec:visc} we discuss the fact that any minimizer of the functional in \eqref{Ju}
is also a viscosity solution, according to Definition~\ref{def:visc}.
This, together with the notion of slab flatness, will allow us to apply the regularity theory
developed in \cite{LN1, LN2} for viscosity solutions.

Finally, in Section \ref{sec-eigen-1} we recall some results concerning
the relation between the characteristic numbers corresponding to two 
complementary cones for $p\neq2$.

%%%%%%%%%%%%%%%%%%%%%%%%%%%
%				
%			 	SECTION													
%																						
%%%%%%%%%%%%%%%%%%%%%%%%%%%
\section*{Notations}
\begin{tabbing}
$C, C_0, C_N, \cdots$ \hspace{1.55cm}       \=\hbox{generic constants,}\\
$\overline U$       \>\hbox{closure of a set} $U$,\\
$\partial U$        \>\hbox{boundary of a set }  $U$,\\
$B_r(x),B_r$\> ball centered at~$x$ with radius~$r>0$, $B_r:=B_r(0)$, \\
$\Gamma$      \> \hbox{the  free boundary} $\partial\{u>0\}$,\\
$\Gamma_{\pm}$      \> $B_1\cap\partial\{u^\pm>0\}$,\\
$\displaystyle\fint$\> mean value integral, \\
$\omega_N$\> volume of unit ball,\\
$\lambda(u)$\>$\lambda_+^p\,\chi_{\{u>0\}} +\lambda_-^p\,\chi_{\{u\le 0\}}$,\\
$\Lambda=\lambda_+^p-\lambda^p_-$\> Bernoulli constant.
\end{tabbing}

%%%%%%%%%%%%%%%%%%%%%%%%%%%
%				
%			 	SECTION													
%																						
%%%%%%%%%%%%%%%%%%%%%%%%%%%
\section{Technicalities}\label{sec:tec}

In this section we gather some basic facts that we shall use in the forthcoming sections.
One of the important results to be proved is the coherence
estimate~\eqref{coherence u}.
In the case~$p=2$ this estimate was showed in~\cite{ACF} (see Theorem~4.1 there),
and the proof uses the Poisson representation formula, that we do not have
for $p\not =2$. However, a combination of the methods
from \cite{ACF}, \cite{DiB-M} and \cite{Giaq} will give the result.

\subsection{Some basic properties of the local minimizers of $J$}

In the proposition to follow all claims are valid in any dimension.  
\begin{prop}\label{prop:tec}
Let $u\in W^{1,p}(\Omega)$ be a local minimizer of $J$. Then
\begin{itemize}
\item[P.1] $\Delta_p u^\pm\geq 0$ in the sense of distributions and $\Delta_p u=0$ in $\{u>0\}\cup \{u<0\}$,
\item[P.2] there is $c_0>0$ such that if
$$\limsup_{r\to 0}\frac{|B_r(x_0)\cap \{u<0\} |}{|B_r(x_0)|}\le c_0, \quad x_0\in \Gamma,$$
then $u$ has linear growth near $x_0$ depending only on $\frac1{c_0}$ times some tame constant,
\item[P.3] $\nabla u\in L^q$ locally, for any finite $q>1$, 
and $u$ is locally log-Lipschitz continuous,
\item[P.4] 
for any $D\Subset \Omega$ there exist~$\bar{r}>0$ and $C>0$ depending on 
$p, \sup |u|, \dist(D, \partial \Omega)$ such that for any~$x_0\in\Gamma\cap B_1$
\begin{equation}\label{coherence u}
\left|\fint_{\partial B_r(x_0)}u\right|\le Cr, \quad {\mbox{ for any }}r\le \bar{r}.
\end{equation}
\end{itemize}
\end{prop}

\begin{remark}
Note that P.4 in Proposition~\ref{prop:tec} says that
either both ~$\frac{1}{r}\fint_{\partial B_r(x_0)}u^+$
and~$\frac{1}{r}\fint_{\partial B_r(x_0)}u^-$ go to~$+\infty$ as $r\to0$
or they both remain bounded.
\end{remark}

\begin{remark}
We stress on the fact that the results in Proposition~\ref{prop:tec}
hold in any dimension
\end{remark}

\begin{proof}
P.1 follows from a standard comparison of $u$ and $u+\epsilon \phi$ for a suitable
smooth compactly supported function $\phi$, and the proof of P.2 can be found in \cite{K1}.

Now we focus on the proofs of P.3 and P.4. 
For this, we observe that it is enough to show that 
\begin{equation}\label{BMO123}
{\mbox{locally }}\, \na u\in BMO. \end{equation}
Indeed, if this is true then $\nabla u\in L^q$ locally, for any $1<q<+\infty$. 
Moreover, the log-Lipschitz estimate follows from 
\cite{Cianchi}, Theorem 3. This proves P.3. 

Also, $u$ is continuous and
$$ \lim_{r\to 0}\fint_{\partial B_r(x_0)}u=0 \quad \mbox{ for any} x_0\in \Gamma.$$
Now, we notice that, for~$\epsilon>0$,
\begin{eqnarray*}\left|
\frac{1}{\epsilon^{N-1}}\int_{B_\epsilon(x)}\nabla u(x)\cdot\frac{x-x_0}{|x-x_0|}\,dx \right|&=&
\left|\frac{1}{\epsilon^{N-1}}\int_{B_\epsilon(x)}\left(\nabla u(x)-
\fint_{B_\epsilon(x)}\nabla u\right)\cdot\frac{x-x_0}{|x-x_0|}\,dx \right|\\
&\le &\epsilon\left(\frac{1}{\epsilon^{N}}\int_{B_\epsilon(x)}\left|
\nabla u(x)- \fint_{B_\epsilon(x)}\nabla u\right|\,dx\right)\longrightarrow 0,
\end{eqnarray*}
as~$\epsilon\to0$, thanks to the BMO estimate in \eqref{BMO123}. Thus
\begin{eqnarray*}
\frac1{r^{N-1}}\int_{\partial B_r(x)}u&=&\int_0^r\left(\frac{d}{dt}\int_{\partial B_1}u(x_0+t\omega)\,d\mathcal{H}^{1}\right)\,dt\\\nonumber
&=&
\int_0^r\frac1{t^{N-1}}\int_{\partial B_t}\na u(x_0+\nu)\cdot \nu \,d\mathcal{H}^{1}\,dt=\\\nonumber
&=&
\int_0^r\frac1{t^{N-1}}\frac{d}{dt}\left(\int_{B_t(x_0)}\na u(x)\cdot \frac{x-x_0}{|x-x_0|}\,dx\right)\,dt=\\\nonumber
&=&\frac1{r^{N-1}}\int_{B_r(x_0)}\na u(x)\frac{x-x_0}{|x-x_0|}\,dx +(N-1)\int_0^r\frac1{t^N}\int_{B_t(x_0)}\na u(x)\frac{x-x_0}{|x-x_0|}\,dx\, dt\\\nonumber
&=&\frac1{r^{N-1}}\int_{B_r(x_0)}\left[\na u(x)-\fint_{B_r(x_0)}\na u\right]\frac{x-x_0}{|x-x_0|}\,dx\\\nonumber
&& +(N-1)\int_0^r\frac1{t^N}\int_{B_t(x_0)}\left[\na u(x)-\fint_{B_t(x_0)}\na u\right]\frac{x-x_0}{|x-x_0|}\,dx\, dt.
\end{eqnarray*}
Therefore, the BMO estimate in \eqref{BMO123}  yields
\begin{equation*}
\left|\frac1{r^{N-1}}\int_{\partial B_r(x)}u\right|\le 3r \|\na u\|_{BMO},
\end{equation*}
which gives the desired result in P.4.

Hence, it remains to show \eqref{BMO123}, that is that locally $\na u \in BMO$. 
In order to provce it, fix~$R\ge r>0$ and
let $v$ be the solution of $\Delta_p v=0$ in $B_{2R}(x_0)$ and
$v=u$ on $\partial B_{2R}(x_0)$. If follows from \cite{DP2} p. 100 that
$$\int_{B_{2R}(x_0)}|\na(u-v)|^p\le CR^N,$$
for some tame constant $C>0$. Notice that, by H\"older inequality, 
\begin{equation}\label{3.2ter}
\int_{B_{2R}(x_0)}|\na(u-v)|^2\le CR^N,
\end{equation}
up to renaming $C$. 

Now, we denote by
$$(\na u)_{x_0,\rho}:=\fint_{B_\rho(x_0)}\nabla u, $$
and we observe that, using H\"older inequality,
\begin{equation}\begin{split}\label{3.2bis}
\int_{B_r(x_0)}|(\na v)_{x_0,r}-(\na u)_{x_0,r}|^2 &\,=
\int_{B_r(x_0)}\left|\frac{1}{|B_r|}\left(\int_{B_r(x_0)}\na v-\na u\right)\right|^2 \\
&\,\le \left(\int_{B_r(x_0)}|\nabla v-\nabla u|\right)^2\\
%&\le &\frac{1}{r^2}\left(\left(\int_{B_r(x_0)}|\nabla v-\nabla u|^2\right)^{1/2}r\right)^2\\
&\,\le \int_{B_r(x_0)}|\nabla v-\nabla u|^2.
\end{split}\end{equation}
Furthermore, we have the following Campanato growth  type estimate (see  \cite{DiB-M} Theorem 5.1)
\begin{equation}\label{Manf-camp-0}
\int_{B_r(x_0)}|\na v-(\na v)_{x_0,R}|^2\lesssim 
\left(\frac rR\right)^{N+\alpha}\int_{B_R(x_0)}|\na v-(\na v)_{x_0,R}|^2,
\end{equation}
where the symbol $\lesssim$ means that the inequality is true up to a positive constant.  
%%%
\allowdisplaybreaks

Therefore, using \eqref{3.2ter}, \eqref{3.2bis} and \eqref{Manf-camp-0}, we have
\begin{eqnarray*}\nonumber
\int_{B_r(x_0)}|\na u-(\na u)_{x_0,r}|^2&\lesssim& \int_{B_r(x_0)}|\na u-\na v|^2+\int_{B_r(x_0)}|\na v-(\na v)_{x_0,r}|^2\\\nonumber
&&+\int_{B_r(x_0)}|(\na v)_{x_0,r}-(\na u)_{x_0,r}|^2\\\nonumber
&\lesssim&\int_{B_r(x_0)}|\na u-\na v|^2+\int_{B_r(x_0)}|\na v-(\na v)_{x_0,r}|^2\\\label{Campanato}
&\lesssim& \int_{B_r(x_0)}|\na u-\na v|^2+\left(\frac rR\right)^{N+\alpha}\int_{B_R(x_0)}|\na v-(\na v)_{x_0,R}|^2\\\nonumber
&\lesssim& \int_{B_r(x_0)}|\na u-\na v|^2\\\nonumber
&&+\left(\frac rR\right)^{N+\alpha}\left[
\int_{B_R(x_0)}|\na v-\na u|^2+ \int_{B_R(x_0)}|\na u-(\na u)_{x_0,R}|^2\right]\\\nonumber
&&+\left(\frac rR\right)^{N+\alpha}\int_{B_R(x_0)}|(\na u)_{x_0,R}-(\na v)_{x_0,R}|^2\\\nonumber
&\lesssim& \int_{B_r(x_0)}|\na u-\na v|^2\\\nonumber
&&+\left(\frac rR\right)^{N+\alpha}\left[
\int_{B_R(x_0)}|\na v-\na u|^2+ \int_{B_R(x_0)}|\na u-(\na u)_{x_0,R}|^2\right]\\\nonumber
&\lesssim& \int_{B_R(x_0)}|\na u-\na v|^2+
\left(\frac rR\right)^{N+\alpha}
 \int_{B_R(x_0)}|\na u-(\na u)_{x_0,R}|^2\\\nonumber
&\lesssim& (R)^N+
\left(\frac rR\right)^{N+\alpha}
 \int_{B_R(x_0)}|\na u-(\na u)_{x_0,R}|^2.\\\nonumber
\end{eqnarray*}
%where, in order to get \eqref{Campanato}, we used Campanato growth  type estimate  from  \cite{DiB-M} Theorem 5.1.

Now, we define
$$ \psi(r):=\sup\limits_{t\leq r} \int\limits_{B_t(x_0)}|\na u-(\na u)_{x_0,t}|^2.$$
It follows from \cite{DiB-M} that
$$\psi(r)\le A\left(\frac rR\right)^{N+\alpha}\psi(R)+B R^N$$
for some positive constants $A$, $B$ and $\alpha$. Applying
Lemma 2.1 from \cite{Giaq} Chapter 3  we  conclude that there exist $c>0$ 
and $R_0>0$ such that
$$\psi(r)\leq cr^N\left(\frac{\psi(R)}{R^N}+B\right)$$
for all $r\le R\le R_0$,
and hence
$$ \int_{B_r(x_0)}|\na u-(\na u)_{x_0,r}|^2 \le C r^N$$
for some tame constant $C>0$.
This shows that~$\nabla u$ is locally BMO and concludes the proof of \eqref{BMO123}. 
The proof of Proposition~\ref{prop:tec} is then complete.
\end{proof}

As a consequence, we have: 

\begin{corollary}\label{coro3.4}
Let $u\in W^{1, p}(\Omega)$ be a local minimizer of $J$. Then for 
any subdomain $D\Subset \Omega$ there is a constant $C>0$ depending on $p, \sup |u|$
and $\dist(D, \partial \Omega)$ such that 
\begin{equation}\label{cod-ball}
\left|\fint_{B_r(x_0)}u\right|\le Cr \quad {\mbox{ for any }} x_0\in \fb{u} \, 
{\mbox{ and }}\, r>0 \, {\mbox{ such that }}\, B_r(x_0)\subset D.
\end{equation}
\end{corollary}

\subsection{A remark on the two-phase problem}\label{rem:one phase}
We can write the functional in~\eqref{Ju} as
$$ J(u)= \int_{\Omega}|\nabla u|^p +\Lambda\chi_{\{u>0\}}
+\lambda^p_-|\Omega|,$$
with~$\Lambda:=\lambda^p_+-\lambda^p_->0$.
Notice that the last term does not affect the minimization problem,
and so if~$u$ is a minimizer for~$J$, then it is also a minimizer for
\begin{equation}\label{erdfgvxb}
\tilde{J}(u):= \int_{\Omega}|\nabla u|^p +\Lambda\chi_{\{u>0\}}.
\end{equation}
Observe that the free boundary~$\partial\{u>0\}\cap\partial\{u\le 0\}$
for the minimizer~$u$ of~$J$ coincides with~$\partial\{u>0\}$ if $\Lambda>0$, 
see e.g. Subsection 3.4 in \cite{DK}.

\subsection{Alt-Caffarelli-Friedman monotonicity formula}
Here we recall a result obtained in~\cite{CKSh}, see in particular Lemmata 2.2 and 2.3
there.
\begin{theorem}\label{Luis-mon}
Let $p=2$ and $u^\pm$ be two continuous subharmonic functions with disjoint
supports in $B_1$ and such that $u^\pm(0)=0$. 

Then we have that
$$\phi_2'(r, u, x_0)\geq \frac2 r\phi_2(r, u, x_0)(\gamma(\Gamma_+)+\gamma(\Gamma_-)-2), $$
where $\varphi_2$ has been introduced in \eqref{varphi} and 
$$ \gamma(\Gamma_\pm)\left( \gamma(\Gamma_\pm) +N-2\right) = 
\inf_{ v\in W_0^{1,2}(\Gamma_\pm) } 
\frac{ \int_{\Gamma_\pm} |\nabla_\theta v|^2 }{\int_{\Gamma_\pm} v^2}. $$

Furthermore, let $\gamma(r):=\gamma(\Gamma_+)+\gamma(\Gamma_-)-2$.
Then  $\gamma(r) \ge  0$ for all small $r$.
Moreover the strict inequality holds unless $\Gamma_\pm^*$ are both half-spheres.
In particular if any of the $\Gamma^*_\pm$ 
digresses from being a half-spherical cap by an area-size of $\epsilon$, say, then
$$\gamma(r) > C\epsilon^2,$$
for some $C>0$. Here $E^*$ stands for the spherical  symmetrisation of $E$.
\end{theorem}

We will use here only the two dimensional version of Theorem \ref{Luis-mon}.

%%%
%%% Subsection capacity
\subsection{Some estimates for capacity}
In this  section we gather some well-known facts about the capacity on the plane 
and the one dimensional Hausdorff measure. So, we fix $N=2$ and, for $\rho>0$, we define 
\begin{equation}
\mathcal H^1_\rho(E):=\inf\sum_i r_i,
\end{equation}
where the infimum is taken over all the coverings of $E\subset \R^2$ by countably many balls of radii
$r_i\le \rho$. Clearly $\mathcal H_\rho^1(E)$ is a decreasing function of $\rho$, 
hence if $\mathcal H^1(E):=\lim_{\rho\to 0}\mathcal H^1_\rho(E)$ exists
then it is  called the one dimensional Hausdorff measure of $E$.
It is also useful to define the set function $\mathcal H^1_\infty$ called the Hausdorff content. 

Throughout this paper the 
$C_{1, \ell}$ capacity, defined in \cite{Adams} page 20, is denoted by $\ca_\ell$.
Let 
$\ca_{\ell}(E,Q)$ be the $\ell$ capacity of $E\subset Q$ where 
$Q\subset \R^2$ is a square and $1<\ell<2$.

We have the following lower estimate for the capacity in terms of the Hausdorff content, see e.g. 
Corollary 5.1.14,  inequality (5.1.3) in \cite{Adams}:
\begin{equation}\label{lower-cap-H-0}
\ca_\ell(E, Q)\ge \ca_\ell (E, \R^2)\ge A (\mathcal H_\infty^1(E))^{2-\ell}
\end{equation}
for a tame constant $A>0$.

It is convenient to formulate a version of \eqref{lower-cap-H-0} replacing the 
Hausdorff content with the measure $\mathcal H^1$.    
For this, let $E:=\partial\{v>0\}$, 
for some continuous function $v\in C(Q)$ such that $\fb v$ is connected,  
the centre of the square  $Q$ belongs to 
$\fb{v}$ and $\fb{v}\cap \partial Q\not=\emptyset$. If $\mathcal L_0$ is a line 
passing though the centre of $Q$ and a point on $\fb{v}\cap \partial Q$ then 
the $\mathcal H^1$ measure of the projection of $\fb{v}$ on $\mathcal L_0$ 
is at least $\frac{\text{diam} Q}{2\sqrt{2}}$.
Let now $\sigma>0$ be such that 
$\mathcal H^1(E)>\sigma$. Observe that there is $\rho_0>0$ such that $\sum_ir_i\ge \sigma$ if
$r_i\leq \rho_0$, for all coverings of $E$ by countably many balls of radii $r_i\le \rho_0$.
Moreover, for all the other coverings we have that
$\sum_{i}r_i\ge\sum_{i=1}^{[\frac{\sigma}{2\rho_0}]} \rho_0  \ge\frac\sigma4$.
Thus, choosing $\sigma:=\frac1{2}\mathcal H^1(E)$, we get from \eqref{lower-cap-H-0} 
\begin{equation}\label{lower-cap-H}
\ca_\ell(E, Q)\ge \ca_\ell (E, \R^2)\ge A \left(\frac{\mathcal H^1(E)}{8}\right)^{2-\ell}.
\end{equation}

We will also need another lower estimate for the capacity, see e.g. \cite{Frehse} page 5:
\begin{equation}\label{lower-cap-L}
\ca_\ell(E, \R^2)\ge A |E|^{1-\frac\ell 2},
\end{equation}
where $|E|$ is the Lebesgue measure on the plane.

Finally we state the Poincar\'e inequality for $v\in W^{1, \ell}$:
there is a tame constant $c>0$ such that 
\begin{equation}\label{Poinc}
\int_{D}|v|^{\ell}\le \frac{c}{\ca_{\ell}(\{v=0\},D)}\int_{D}|\nabla v|^{\ell},
\end{equation}
where $D$ is a ball or a square, see \cite{Frehse} pages 15-16.

\subsection{Gehring's Lemma}\label{ap:geh}

Here we recall the Gehring's result on the higher integrability, 
see \cite{Giaq}, Proposition 1.1, page 122.
\begin{prop}\label{prop-Geh}
Let $Q$ be a square and $r>q\ge1$. 
Suppose that $f\in L^r(Q)$, and that
\begin{equation}\label{eq-Geh0}
\fint\limits_{Q_R(x_0)}g^q\le b \left(\fint\limits_{Q_{2R}(x_0)}g\right)^q+\fint\limits_{Q_{2R}(x_0)}f^q+\theta\fint\limits_{Q_{2R}(x_0)}g^q
\end{equation}
for each $x_0\in Q$ and each $R<\min\{\frac12\dist(x_0, \partial Q), R_0\}$, where
$R_0$, $b$ and $\theta$ are constants with $b>1$, $R_0>0$ and $0\leq \theta<1$. 

Then there exist $\e>0$ and $c>0$ such that
$g\in L_{\rm loc}^p(Q)$ for $p\in [q, q+\e)$ and
\begin{equation}\label{eq-Geh1}
\left(\fint\limits_{Q_R}g^p\right)^{\frac1p}\le c\left\{
\left(\fint\limits_{Q_{2R}}g^q\right)^{\frac1q}+
\left(\fint\limits_{Q_{2R}}f^pdx\right)^{\frac1p}\right\},
\end{equation}
for any $R<R_0$ such that $Q_{2R}\subset Q$, 
where $c$ and $\e$ are positive constants depending on $b,\theta,q$ and $r$.
\end{prop}

%%%%%%%%%%%%%%%%%%%%%%%%%%%
%				
%			 	SECTION													
%																						
%%%%%%%%%%%%%%%%%%%%%%%%%%%
\section{Proof of Theorem~\ref{TH2}}\label{s1}

\subsection{Step 0: Heuristic discussion}
We will prove Theorem~\ref{TH2} using a contradiction argument.
That is, we assume that there exist~$p_j\to 2$, with $p_j>2$, $x_j\in\Gamma$ and~$r_j\to 0$, 
as $j\to+\infty$, such that $h(x_j, r_j)> h_0\,r_j$ and
\begin{equation}\label{contradiction}
\varphi_{p_j}(r_j,u,x_j)>\varphi_{p_j}(3r_j,u,x_j).
\end{equation}
We set
\begin{equation}\label{Sj}
S^\pm_j:=\left(\fint_{B_{3 r_j}(x_j)}|\nabla u^\pm(y)|^{p_j}\,dy\right)^{1/p_j},
\end{equation}
%$$\rred{S^\pm_j:=\left(\fint_{B_{3 r_j}(x_j)}\frac{|\nabla u^\pm(y)|^{p_j}}{|x|^{N-2}}\,dy
%\right)^{1/p_j},}$$\mpar{for higher dimension}
and introduce the scaled functions
\begin{equation}\label{uj}
u^{\pm}_j(x):=\frac{u^{\pm}(x_j+r_jx)}{r_j\,S_j^\pm}.
\end{equation}
By construction we have
\begin{equation}\label{bbb00}
\|\nabla u^{\pm}_j\|_{L^{p_j}(B_3)}=3^{2/p_j}\le 3.
\end{equation}
Hence, from~\eqref{contradiction} we deduce that
\begin{equation}\label{pr-probe}
 \varphi_{p_j}(1,u_j,0)>\varphi_{p_j}(3,u_j,0)=1,
 \end{equation}
or equivalently
\begin{equation}\label{bb0}
\int_{B_1}|\nabla u^+_j|^{p_j}\,\int_{B_1}|\nabla u^-_j|^{p_j}>1.
\end{equation}
Thanks to the uniform bound \eqref{bbb00} we can extract a subsequence that weakly converges
to some $u^\pm_0\in W^{1,2}(B_1)$. Consequently,
from the semicontinuity of the Dirichlet's integral we have that
\begin{equation}\label{phi-lowlim}
\varliminf\varphi_{p_j}(1,u_j,0)\ge\varliminf \varphi_{p_j}(3,u_j,0)\ge \phi_2(3, u_0, 0).
\end{equation}
In order to handle the limit on the left hand side we need strong convergence of $\nabla u_j^\pm$ in, say,
$L^2(B_1)$ (e.g. it will suffice to have \textit{uniform higher integrability}
of $\{\nabla u_j\}$, for instance  $|\nabla u_j|\in L^q(B_1)$ for some fixed
$q>2$, which we will prove using the Gehring's lemma). Suppose for a moment that this is true,
then passing to the limit in (\ref{phi-lowlim}) we infer the inequality
\begin{equation}\label{lim-phi-0}
\phi_2(1,u_0, 0)\ge \phi_2(3, u_0, 0).
\end{equation}
Note that thanks to the uniform convergence $u_j\to u_0$, as $j\to+\infty$, 
(due to the estimate $|\na u|\in L^{q}_{loc}$, with $q>2$, and the Sobolev embedding), 
we obtain that \eqref{delta cond} translates to
\begin{equation}\label{non-flat-0}
h(0, 1)\ge h_0.
\end{equation}
Furthermore, from P.1 in Proposition \ref{prop:tec} we have that $\Delta_{p_j}u_j^\pm\ge0$,
and this translates to $\Delta u^\pm_0\ge0$ in view of the $W^{1,q}$ estimate for $q>2$.

If both the functions $u_0^\pm$ do not vanish (i.e. both $u_j^+$ and $u_j^-$ are
\textit{non-degenerate})
then $u^\pm_0$ are admissible functions in Theorem~\ref{Luis-mon},
and we infer  from \eqref{lim-phi-0} that
$u_0=u^+_0-u_0^-$ is a two-plane solution in $B_3\setminus B_1$.
Consequently, employing some standard unique continuation
results for harmonic functions we shall conclude that $u_0$ is a two-plane solution in $B_3$ which,  however, will be
in contradiction with \eqref{non-flat-0} and the proof will follow.
\medskip 

Now we begin with the actual proof of Theorem \ref{TH2}.
It is convenient to split the proof into a number of steps,
which in combination shall yield the proof of Theorem \ref{TH2}. In Step 1 below
we prove that the scaled functions defined in \eqref{uj} 
remain uniformly non-degenerate  in $L^2(B_2)$. Step
2, which is the most technical one,  takes care of
the higher integrability of the gradient of the scaled functions $\nabla u_j$, allowing us to
pass to the the limit in \eqref{phi-lowlim}. To do so we employ the Gehring's Lemma 
(recall Proposition \ref{prop-Geh}) 
and the Caccioppoli's inequality. One more technical issue that arises here
is to establish a Poincar\'e type estimate for the scaled functions $u^\pm_0$.
In Step 3 and Step 4 we perform a
gap filling argument based on some ideas from the unique continuation theory,
allowing us to extend the linearity of $u_0$ from $B_3\setminus B_1$ into $B_1$.

\medskip
\subsection{Step 1: Non-degeneracy}
In order to take the limit of the scaled functions~$u^{\pm}_j$ as~$j\rightarrow+\infty$
(recall \eqref{uj}),
we need to ensure that both~$u_j^+$ and~$u^-_j$ do not vanish identically.
Lemma to follow provides a lower bound in term of $L^p$ integrals.

\begin{lemma}\label{lem-nondeg}
Let~$u_j^{\pm}$ be as in~\eqref{uj}. Then, there exists~$C_0>0$
independent of~$j$ such that
$$ \int_{B_2}|u^+_j|^{p_j}\,\int_{B_2}|u^-_j|^{p_j}\ge C_0.$$
\end{lemma}

\begin{proof}
From the scaling properties of the operator $\Delta_p$ it follows that 
$u_j^+$ is~$p_j$-subharmonic in~$B_3$. 
Therefore, we have that, for any~$\psi\in C^1_0(B_3)$, with $\psi\ge0$,
\begin{equation}\label{subsol}
\int_{B_3}|\nabla u^+_j|^{p_j-2}\nabla u_j^+\cdot\nabla\psi \le 0.
\end{equation}
Now, we consider a cutoff function~$\eta\in C^\infty(B_3)$ such that~$\eta\ge0$, $\eta\equiv 0$ in~$B_3\setminus B_2$ and~$\eta\equiv 1$ in~$B_1$, and
we take $\psi:=u^+_j\,\eta^{p_j}$ in~\eqref{subsol}. We obtain
$$ \int_{B_3}|\nabla u^+_j|^{p_j}\eta^{p_j} + p_j\,\int_{B_3}|\nabla u^+_j|^{p_j-2} u_j^+\,\eta^{p_j-1} \nabla u^+_j\cdot\nabla\eta\le 0,$$
which implies, using H\"older's inequality,
\begin{eqnarray*}
\int_{B_3}|\nabla u^+_j|^{p_j}\eta^{p_j} &\le & p_j\,\int_{B_3}\left(|\nabla u^+_j|^{p_j-1}\eta^{p_j-1}\right)\left(u^+_j|\nabla\eta|\right) \\
&\le & p_j\, \left(\int_{B_3}|\nabla u^+_j|^{p_j}\eta^{p_j}\right)^{\frac{p_j-1}{p_j}}\left(\int_{B_3}|u^+_j|^{p_j}\,|\nabla\eta|^{p_j}
\right)^{\frac{1}{p_j}}.
\end{eqnarray*}
This gives that
$$ \int_{B_3}|\nabla u^+_j|^{p_j}\eta^{p_j}\le p_j^{p_j}\int_{B_3}|u^+_j|^{p_j}\,|\nabla\eta|^{p_j}.$$
Therefore, recalling the properties of~$\eta$, we obtain that
\begin{equation}\label{bb1}
\int_{B_1}|\nabla u^+_j|^{p_j} \le \int_{B_3}|\nabla u^+_j|^{p_j}\eta^{p_j} \le
p_j^{p_j}\int_{B_3}|u^+_j|^{p_j}\,|\nabla\eta|^{p_j} \le C\,p_j^{p_j}
\int_{B_2} |u^+_j|^{p_j},
\end{equation}
for some~$C>0$ independent of~$j$.

Notice that a similar result holds if we substitute~$u^+_j$ with~$u^-_j$
in the previous computations. Namely,
$$ \int_{B_1}|\nabla u^-_j|^{p_j} \le C\,p_j^{p_j}
\int_{B_2} |u^-_j|^{p_j}.$$
Combining this and~\eqref{bb1} and using~\eqref{bb0}, we get
\begin{eqnarray*}
\int_{B_2} |u^+_j|^{p_j}\,\int_{B_2} |u^-_j|^{p_j} &\ge &
\frac{1}{C^2\,p_j^{2p_j}}\,\int_{B_1}|\nabla u^+_j|^{p_j}\,\int_{B_1}|\nabla u^-_j|^{p_j} \\
&>& \frac{1}{C^2\,p_j^{p_j}} \\
&\ge & C_0,
\end{eqnarray*}
for a suitable~$C_0>0$ independent of~$j$ (recall that~$2<p_j<p_0$).
This concludes the proof of Lemma~\ref{lem-nondeg}.
\end{proof}

\medskip
\subsection{Step 2: Higher integrability}
The next result is based on the Gerhing's Lemma (see \cite{Giaq} page 122
and Proposition~\ref{prop-Geh} here) and allows us to obtain
higher integrability of~$\nabla u^{\pm}_j$ and thus to justify
the passage to the limit and infer \eqref{lim-phi-0}.

\begin{lemma}\label{lem-Gehring}
Let~$u_j^{\pm}$ be as in~\eqref{uj}. Then there exist~$q>2$ and~$C>0$
independent of~$j$ such that
$$ \|\nabla u^{\pm}_j\|_{L^q(B_1)}\le C.$$
\end{lemma}

\begin{proof}
We first claim that there exists
a universal constant~$\bar{C}>0$ such that, for any square~$Q_{2R}\subset B_3$
(with $R>1$), it holds
\begin{equation}\label{bb55}
\left(\fint_{Q_R} |\nabla u^{\pm}_j|^2\right)^{1/2}\le \bar{C}\left(\fint_{Q_{2R}} |\nabla u^{\pm}_j|^{\ell}\right)^{1/\ell},
\end{equation}
for any fixed $\ell$ satisfying $p_0/2<\ell<2$
(recall that~$p_0$ is the constant in~\eqref{p cond}).
However, one may also  take $\ell:=\frac32$, since here 
it is only important to have $\ell\in(1, 2)$, i.e. the lower order 
norm controls the higher order norm.

We show~\eqref{bb55} only for~$u^+_j$, since the proof for~$u^-_j$
is analogous. We denote by
\begin{equation}\label{ellej}
\ell_j:=\frac{2p_j}{2+p_j},
\end{equation}
that is~$p_j$
is the Sobolev exponent corresponding to~$\ell_j$.
Notice that~$1<\ell_j<p_0/2$, therefore, if~$\ell>p_0/2$ then~$\ell>\ell_j$.
Also,~$\ell_j\to 1$ as~$j\to+\infty$.

So we fix~$\ell$ independent of~$j$ such that~$p_0/2<\ell<2$ and
consider three possibilities:
\begin{itemize}
\item[\bf Case 1)] $Q_{2R}\cap\partial\{u_j^+=0\}\neq\emptyset$ and~$\ca_{\ell}(\{u^+_j=0\},Q_{2R})\ge\delta R^{2-\ell}$, for any~$j$, for some~$\delta>0$ independent of~$j$,
\item[\bf Case 2)] $Q_{2R}\cap\partial\{u_j^+=0\}\neq\emptyset$
but the capacity~$\ca_{\ell}(\{u^+_j=0\},Q_{2R})$ is small,
\item[\bf Case 3)] $Q_{2R}\cap\partial\{u_j^+=0\}=\emptyset$.
\end{itemize}

%\rred{Note  the capacity is scaled as follows $\ca_{\ell}(Q_r)=\ca_{\ell}(Q_1) r^{N-\ell}$.}

\medskip

{\bf Case~1):} We use the fact that~$u^+_j$ is $p$-subharmonic in~$B_{3R}$
(recall P.1 in Proposition~\ref{prop:tec}) to deduce
that, for any~$\psi\in C^1_0(B_{3R})$, with $\psi\ge0$, we have
\begin{equation}\label{bb56}
\int_{B_{3R}}|\nabla u^+_j|^{p_j-2}\nabla u^+_j\cdot\nabla\psi\le 0.
\end{equation}
Now we take a cutoff function~$\eta\in C^\infty(B_{3R})$ such that~$\eta\ge0$,
$\eta\equiv 1$ in~$Q_R$, $\eta\equiv 0$ outside~$Q_{2R}$ and~$|\nabla\eta|\le\frac{C}{R}$ for some~$C>0$.
Then, we choose~$\psi:=u^+_j\,\eta^{p_j}$ in~\eqref{bb56} and we obtain that
$$ \int_{B_{3R}}|\nabla u^+_j|^{p_j}\eta^{p_j} + p_j\,\int_{B_{3R}}|\nabla u^+_j|^{p_j-2} u_j^+\,\eta^{p_j-1} \nabla u^+_j\cdot\nabla\eta\le 0.$$
After applying H\"older's inequality, this yields
$$ \int_{B_{3R}}|\nabla u^+_j|^{p_j}\eta^{p_j} \le p_j^{p_j}\, \int_{B_{3R}}|u^+_j|^{p_j}|\nabla\eta|^{p_j}.$$
Therefore, recalling the properties of~$\eta$, we have
$$ \int_{Q_R}|\nabla u^+_j|^{p_j}\le \frac{C^{p_j}\,p_j^{p_j}}{R^{p_j}}\,\int_{Q_{2R}} |u^+_j|^{p_j},$$
which implies that
\begin{equation}\label{bb600}
\frac{R^{p_j}}{C^{p_j}\,p_j^{p_j}}\,\fint_{Q_R}|\nabla u^+_j|^{p_j}\le 2^2\, \fint_{Q_{2R}} |u^+_j|^{p_j}.
\end{equation}

Rescaling $u^+_j$ and setting
\begin{equation}\label{bb57}
v_j^+(x):=u^+_j(Rx),
\end{equation}
we observe that~$p_j$
is the Sobolev exponent corresponding to~$\ell_j$, see \eqref{ellej},
hence the Sobolev embedding gives that
\begin{equation}\label{sob}
\left(\int_{Q_2}|v_j^+|^{p_j}\right)^{1/p_j}\le C \left(\int_{Q_2}|v^+_j|^{\ell_j}+|\nabla v^+_j|^{\ell_j}\right)^{1/\ell_j}\le
C \left(\int_{Q_2}|v^+_j|^{\ell}+|\nabla v^+_j|^{\ell}\right)^{1/\ell},
\end{equation}
for some~$C>0$ (recall that~$\ell>\ell_j$). Furthermore, using the 
scaling properties of the $\ell-$capacity and applying the Poincar\'e inequality \eqref{Poinc}, we get
\begin{equation}\label{poin}
\left(\int_{Q_2}|v_j^+|^{\ell}\right)^{1/\ell}\le \left(\frac{c}{\ca_{\ell}(\{v_j^+=0\},Q_2)}\right)^{1/\ell}\left(\int_{Q_2}|\nabla v^+_j|^{\ell}\right)^{1/\ell}
\le c_0\left(\int_{Q_2}|\nabla v^+_j|^{\ell}\right)^{1/\ell},
\end{equation}
where~$c_0$ is a positive constant independent of~$j$.

Now, putting together~\eqref{sob} and~\eqref{poin}, we obtain that
\begin{equation}\label{bb58}
\left(\int_{Q_2}|v^+_j|^{p_j}\right)^{1/p_j}\le C(1+c_{0}) \left(\int_{Q_2}|\nabla v^+_j|^{\ell}\right)^{1/\ell}.
\end{equation}

Now we observe that, by~\eqref{bb57} and by making the change of variable~$y=Rx$, we have that
\begin{equation}\begin{split}\label{bb59}
&\int_{Q_2}|v^+_j(x)|^{p_j}\,dx =\int_{Q_2}|u^+_j(Rx)|^{p_j}\,dx \\
&\qquad\qquad = R^{-2}\int_{Q_{2R}}|u^+_j(y)|^{p_j}\,dy = 2^2\,\fint_{Q_{2R}}|u^+_j(y)|^{p_j}\,dy.
\end{split}\end{equation}
Moreover, from~\eqref{bb57} we deduce that
$$ \nabla v^+_j(x)=R\,\nabla u^+_j(Rx),$$
which implies
\begin{equation}\begin{split}\label{bb60}
&\int_{Q_2}|\nabla v^+_j(x)|^{\ell}\,dx = R^{\ell}\,\int_{Q_2}|\nabla u^+_j(Rx)|^{\ell}\,dx \\
&\qquad\qquad = R^{\ell-2}\,\int_{Q_{2R}}|\nabla u^+_j(y)|^{\ell}\, dy
= 2^2\,R^{\ell}\,\fint_{Q_{2R}}|\nabla u^+_j(y)|^{\ell}.
\end{split}\end{equation}
Using~\eqref{bb59} and~\eqref{bb60} into~\eqref{bb58}, we get
$$ 2^{2/p_j}\, \left(\fint_{Q_{2R}}|u^+_j|^{p_j}\right)^{1/p_j}\le C(1+c_{0})\, 2^{2/\ell}\,R\,\left(\fint_{Q_{2R}}|\nabla u^+_j|^{\ell}\right)^{1/\ell}.$$
From this and~\eqref{bb600}, we obtain
\begin{eqnarray*}
\frac{R}{C\,p_j}\, \left(\fint_{Q_R}|\nabla u^+_j|^{p_j}\right)^{1/p_j} &\le &
2^{2/p_j}\, \left(\fint_{Q_{2R}}|u^+_j|^{p_j}\right)^{1/p_j}\\
&\le &  C(1+c_{0})\, 2^{2/\ell}\,R\,\left(\fint_{Q_{2R}}|\nabla u^+_j|^{\ell}\right)^{1/\ell},
\end{eqnarray*}
or equivalently 
$$ \left(\fint_{Q_R}|\nabla u^+_j|^{p_j}\right)^{1/p_j} \le C\, \left(\fint_{Q_{2R}}|\nabla u^+_j|^{\ell}\right)^{1/\ell}, $$
up to renaming constants.

Now, since~$p_j>2$, for any fixed $\ell$ such that $p_0/2<\ell<2$ we have
\begin{equation*}
\left(\fint_{Q_R}|\nabla u^+_j|^{2}\right)^{1/2} \le C\,
\left(\fint_{Q_R}|\nabla u^+_j|^{p_j}\right)^{1/p_j}
\le C \, \left(\fint_{Q_{2R}}|\nabla u^+_j|^{\ell}\right)^{1/\ell},
\end{equation*}
which establishes~\eqref{bb55} in the Case~1).

\medskip

{\bf Case~2):} Suppose that $\ca_{\ell}(\{u^+_j=0\},Q_{2R})<\delta R^{2-\ell}$.
We take the square~$Q_{\frac{3}{2}R}$
and we consider two subcases:
\begin{itemize}
\item[\bf 2a)] $Q_{\frac{3}{2}R}\cap\{u^+_j=0\}=\emptyset$,
\item[\bf 2b)] $Q_{\frac{3}{2}R}\cap\{u^+_j=0\}\neq\emptyset$.
\end{itemize}

In {\bf Case~2a)}, thanks to P.1 in Proposition~\ref{prop:tec} we have that~$u^+_j$
is~$p_j$-harmonic in~$Q_{\frac{3}{2}R}$, and so
\begin{equation}\label{ghfeuiwgiog}
\int_{Q_{\frac{3}{2}R}}|\nabla u^+_j|^{p_j-2}\nabla u^+_j\cdot\nabla\psi=0,
\end{equation}
for any~$\psi\in C^1_0(Q_{\frac{3}{2}R})$.
Now we take a cutoff function~$\eta\in C^\infty(B_3)$ such that~$\eta\ge0$,
$\eta\equiv 1$ in~$Q_R$, $\eta\equiv 0$ outside~$Q_{\frac{3}{2}R}$
and~$|\nabla\eta|\le C/R$ for some positive~$C$.
We also set
$$ \overline{u}^+_j:=\left(\frac32 R\right)^{-2}\int_{Q_{\frac{3}{2}R}}u^+_j(x)\,dx.$$
Therefore, taking~$\psi:=(u^+_j-\overline{u}^+_j)\eta$ in~\eqref{ghfeuiwgiog},
we obtain that
$$ \int_{Q_{\frac{3}{2}R}}|\nabla u^+_j|^{p_j}+
p_j\int_{Q_{\frac{3}{2}R}}|\nabla u^+_j|^{p_j-2}(u^+_j-\overline{u}^+_j)\eta^{p_j-1}
\nabla u^+_j\cdot\nabla\eta =0.$$
So, by H\"older's inequality,
$$ \int_{Q_{\frac{3}{2}R}}|\nabla u^+_j|^{p_j}\eta^{p_j}\le p_j^{p_j}
\int_{Q_{\frac{3}{2}R}}|u^+_j-\overline{u}^+_j|^{p_j}|\nabla\eta|^{p_j},$$
which implies that
$$ \int_{Q_R}|\nabla u^+_j|^{p_j}\le \frac{C^{p_j}\,p_j^{p_j}}{R^{p_j}}
\int_{Q_{\frac{3}{2}R}}|u^+_j-\overline{u}^+_j|^{p_j},$$
thanks to the properties of~$\eta$. Thus
\begin{equation}\label{qwevbdgrvn}
\frac{R^{p_j}}{C^{p_j}\,p_j^{p_j}} \fint_{Q_R}|\nabla u^+_j|^{p_j}\le \left(\frac{3}{2}\right)^2
\fint_{Q_{\frac{3}{2}R}}|u^+_j-\overline{u}^+_j|^{p_j}.
\end{equation}

Now we rescale~$u^+_j$ in the following way: we set~$v^+_j(x):=u^+_j(Rx)$
and~$\overline{v}^+_j:=\frac49\int_{Q_{\frac{3}{2}}}v^+_j(x)\,dx$. Notice that
\begin{equation}\label{opierhtj}
\overline{v}^+_j=\left(\frac{2}{3}\right)^2\int_{Q_{\frac{3}{2}}}u^+_j(Rx)\,dx=
\left(\frac{2}{3R}\right)^2\int_{Q_{\frac{3}{2}R}}u^+_j(y)\,dy=\overline{u}^+_j.
\end{equation}
From the Sobolev embedding and the Poincar\'e's inequality we get
\begin{equation}\begin{split}\label{ouebsbcvsdg}
\left(\int_{Q_{\frac{3}{2}}}|v^+_j-\overline{v}^+_j|^{p_j}\right)^{1/p_j}
 \le &\, C\,\left(\int_{Q_{\frac{3}{2}}}|v^+_j-\overline{v}^+_j|^{\ell_j}
+|\nabla v^+_j|^{\ell_j}\right)^{1/\ell_j}\\
 \le &\, C\,\left(\int_{Q_{\frac{3}{2}}}|v^+_j-\overline{v}^+_j|^{\ell}
+|\nabla v^+_j|^{\ell}\right)^{1/\ell}\\
\le &\, C\,\left(\int_{Q_{\frac{3}{2}}}|\nabla v^+_j|^{\ell}\right)^{1/\ell},
\end{split}\end{equation}
where~$\ell_j$ is given by~\eqref{ellej}, $p_0/2<\ell<2$, and the constant~$C>0$
may vary from line to line but it is independent on~$j$ (recall~\eqref{delta cond}).

Using the change of variable~$y=Rx$ and~\eqref{opierhtj}, we have that
\begin{eqnarray*}
&&\int_{Q_{\frac{3}{2}}}|v^+_j(x)-\overline{v}^+_j|^{p_j}\,dx =
\int_{Q_{\frac{3}{2}}}|u^+_j(Rx)-\overline{v}^+_j|^{p_j}\,dx\\
&&\qquad = R^{-2}\int_{Q_{\frac{3}{2}R}}|u^+_j(y)-\overline{u}^+_j|^{p_j}\,dy
= \left(\frac{3}{2}\right)^2\omega_2 \fint_{Q_{\frac{3}{2}R}}|u^+_j(y)-\overline{u}^+_j|^{p_j}\,dy.
\end{eqnarray*}
Similarly, one can check that
$$ \int_{Q_{\frac{3}{2}}}|\nabla v^+_j|^{\ell}=\left(\frac{3}{2}\right)^2 R^{\ell}\omega_2\,
\fint_{Q_{\frac{3}{2}R}}|\nabla u^+_j|^{\ell}.$$
Inserting the last two formulas into~\eqref{ouebsbcvsdg} we obtain that
$$ \left(\frac{3}{2}\right)^{2/p_j}
\left(\fint_{Q_{\frac{3}{2}R}}|u^+_j-\overline{u}^+_j|^{p_j}\right)^{1/p_j}
\le C\left(\frac32\right)^{2/\ell}R\left(\fint_{Q_{\frac{3}{2}R}}|\nabla u^+_j|^{\ell}\right)^{1/\ell},$$
which, together with~\eqref{qwevbdgrvn}, implies that
$$ \left(\fint_{Q_R}|\nabla u^+_j|^{p_j}\right)^{1/p_j}\le C\,p_j\left(\frac32\right)^{2/\ell}\left(\fint_{Q_{\frac{3}{2}R}}|\nabla u^+_j|^{\ell}\right)^{1/\ell} \le C\,\left(\fint_{Q_{\frac{3}{2}R}}|\nabla u^+_j|^{\ell}\right)^{1/\ell},  $$
up to renaming~$C$. Notice that~$C$ is independent on~$j$,
thanks to~\eqref{p cond} and the fact that~$\ell<2$.
Thus
$$ \left(\fint_{Q_R}|\nabla u^+_j|^{p_j}\right)^{1/p_j}\le
C\,\left(\fint_{Q_{2R}}|\nabla u^+_j|^{\ell}\right)^{1/\ell}.$$
This, together with the fact that~$p_j>2$, implies~\eqref{bb55}
for any~$p_0/2<\ell<2$. This finishes Case 2a).

\medskip

Now we suppose that {\bf Case~2b)} holds true.
Since the $\ell$-capacity of~$\{u^+_j=0\}$ in~$Q_{2R}$ is small relative to $R^{2-\ell}$,
it cannot happen that~$Q_{\frac{3}{2}R}\subseteq \{u^+_j=0\}$,
otherwise we would have a uniform bound from below for the capacity
(see e.g.~\cite{Frehse}).
Therefore, there exists a point~$q\in\partial\{u^+_j>0\}\cap Q_{\frac{3}{2}R}$.
Let~$\Gamma^+_j:=\partial\{u^+_j>0\}$ and~$K\subseteq\{u^+_j=0\}$ be the component
of~$\{u^+_j=0\}$ such that~$q\in\partial K$.

Suppose first that $K$ is the unique component of $\{u^+_j=0\}$ 
such that $K\cap Q_{\frac{3}{2}R}\neq\emptyset$.
Since~$u$ is a minimizer, then it is log-Lipschitz continuous, see Proposition P.3 in 
\ref{prop:tec},
therefore~$u_j$ is continuous.
Hence,
\begin{itemize}
\item[\bf $2b_1$)] either~$\partial K\cap\partial Q_{2R}\neq\emptyset$, see Figure 1(A),
\item[\bf $2b_2$)] or~$\partial K\Subset Q_{2R}$, see Figure~1(B).
\end{itemize}

\begin{figure}
        \centering
                \begin{subfigure}[b]{0.48\textwidth}
                                \includegraphics[width=\textwidth]{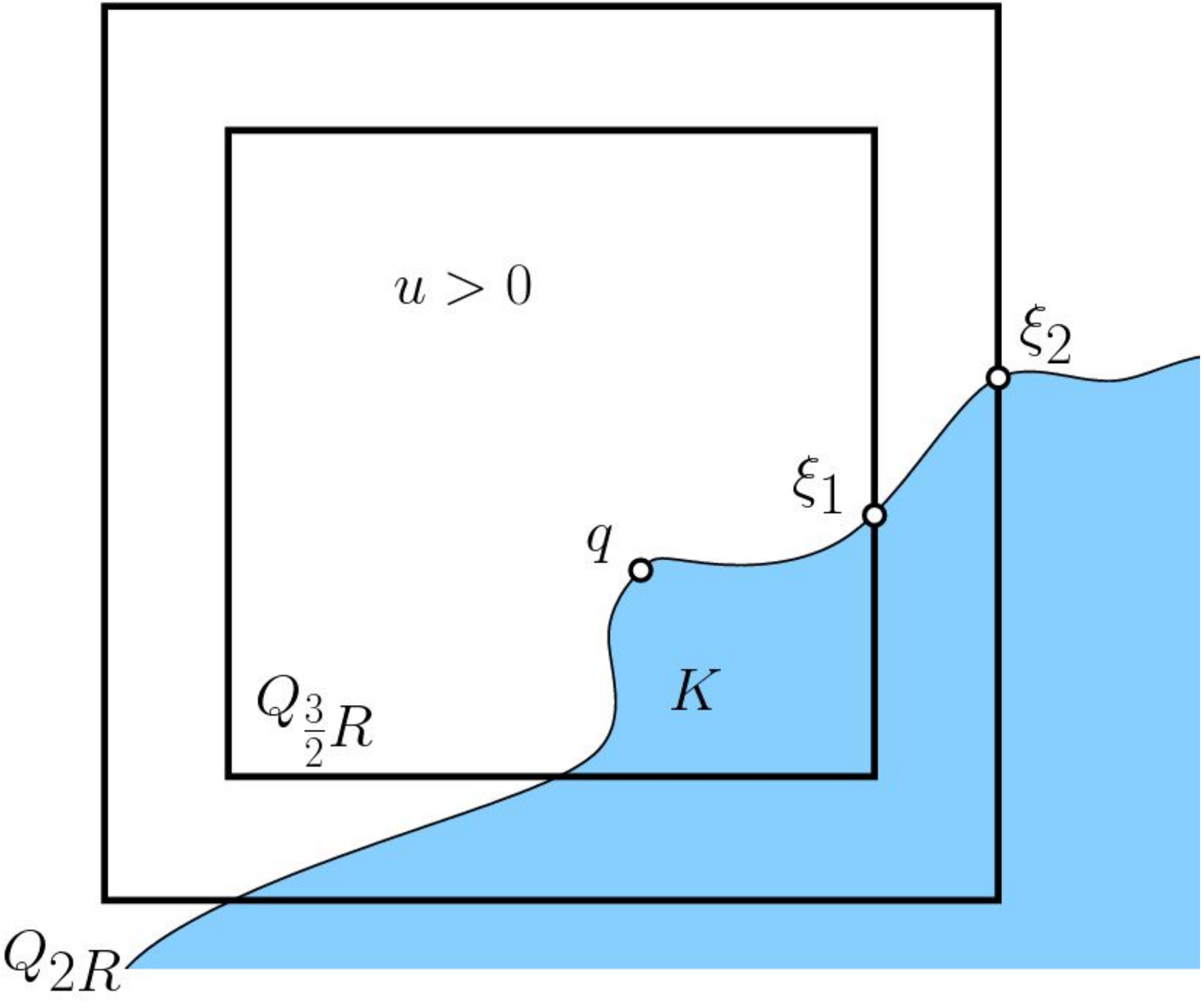}
                \caption{Case $2b_1$}
                       \end{subfigure}
                \begin{subfigure}[b]{0.48\textwidth}
                \includegraphics[width=\textwidth]{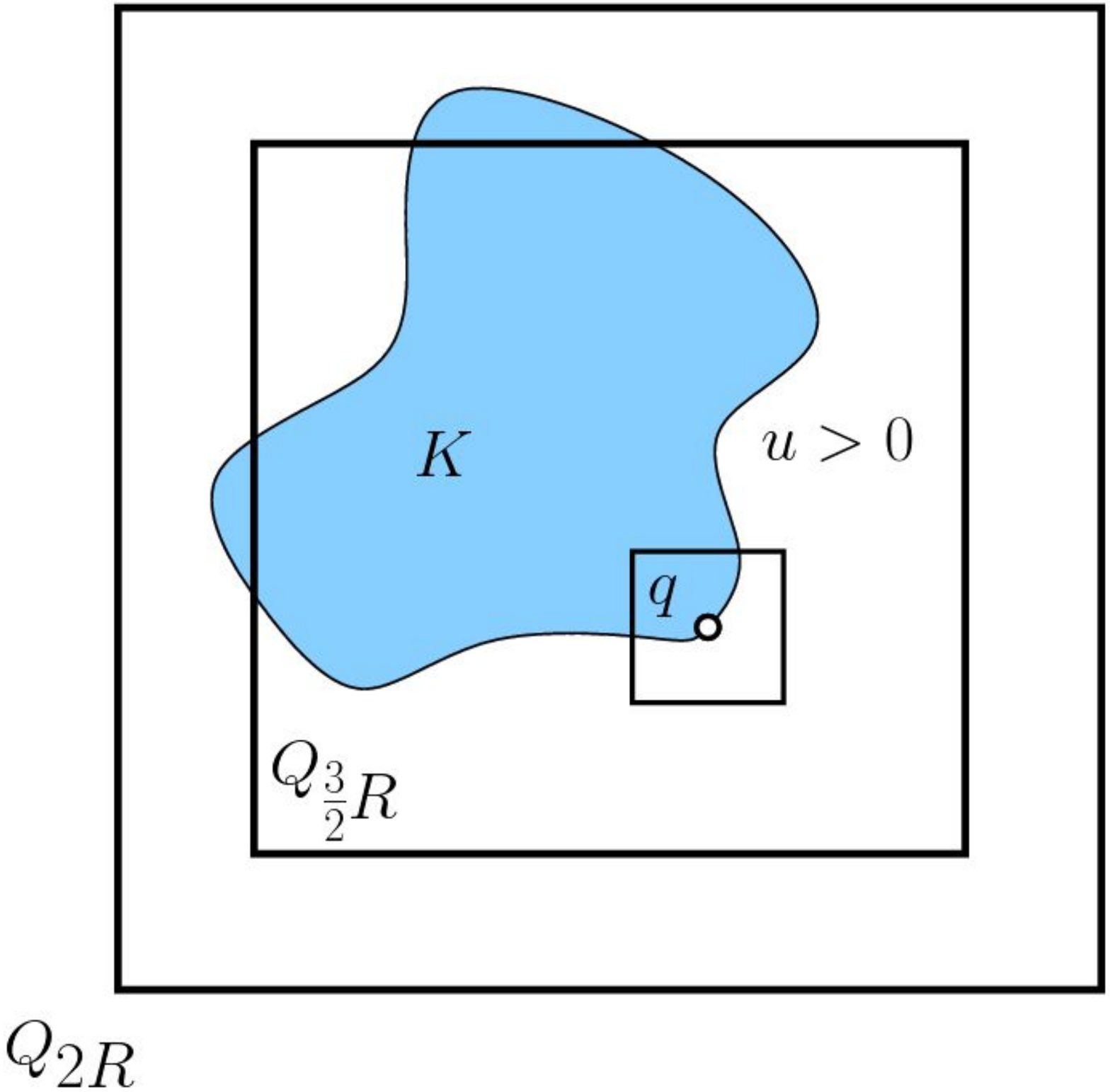}
                %{fig2b2.pdf}
                \caption{Case $2b_2$}
        \end{subfigure}
                \caption{The two sub-cases of Case 2b}
                \end{figure}

\medskip

In {\bf Case~$2b_1$)}, that is when $\partial K\cap\partial Q_{2R}\neq\emptyset$, we have
that
\begin{equation}\label{per big}
\mathcal H^1(\partial \{u^+_j=0\}\cap (Q_{2R}\setminus Q_{\frac32 R}) )\ge\frac{R}{8},
\end{equation}
since~$u_j$ is a continuous functions.
Indeed, let $\xi_1$ and $\xi_2$ be the intersection points of $\partial K$ with
$\partial Q_{\frac32 R}$ and $\partial Q_{2R}$, respectively, and
let $K_0$ be the orthogonal projection of $\tau:=\partial K\cap (Q_{2R}\setminus Q_{\frac32 R})$
on the line joining $\xi_1$ and $\xi_2$.  We consider a
covering of $\tau$, namely $\tau\subset\cup_{i\in I} B_{i}(x_i)$,
such that~$\hbox{diam} B_i(x_i)<\epsilon$
for every $i\in I$.
Hence, denoting by $\bar{x}_i$, $i\in I$,
the projection of $x_i$ on the line that joins $\xi_1$ and $\xi_2$,
we find a covering for $K_0$,
that is~$K_0\subset\cup_{i\in I} B_i(\bar{x}_i)$,
with $\hbox{diam} B_i(\bar{x}_i)<\epsilon$.
Consequently,
$$ \mathcal H^1_\epsilon(\tau)\ge \mathcal H^1_\epsilon(K_0)=\inf \sum_{i\in I} \hbox{diam} B_i(\bar{x}_i)
\ge \frac{R}{8},$$
where the infimum is taken over all the coverings of~$K_0$ such
that~$\hbox{diam} B_i(\bar{x}_i)<\epsilon$.
Hence, sending~$\epsilon$ to zero we obtain~\eqref{per big}.

We notice that~\eqref{per big} gives a lower bound of the capacity, thanks to \eqref{lower-cap-H},
and so we conclude as in Case~1).

In {\bf Case~$2b_2$)}, that is when $\partial K\Subset Q_{2R}$,
we recall Subsection \ref{rem:one phase} in order
to conclude that the free boundary is given just by $\partial\{u>0\}$.

That said, we observe that if~$u_j\le 0$ inside~$K$ and~$u_j\le0$ outside,
then actually~$u_j\equiv0$ inside~$K$,
since~$u_j=0$ on~$\partial K$ and it is~$p_j$-subharmonic inside.
Thus
\begin{equation}\label{positivity}
u_j\ge 0 \quad {\mbox{ in }}Q_{\frac{3}{2}R}.
\end{equation}
Thus, we can consider the pure one-phase minimization problem~\eqref{erdfgvxb}
in $Q_{\frac{3}{2}R}$ (recall that~$\Lambda>0$).

Now, if~$Q_{\frac{3}{4}R}$ is contained in the set~$K$,
then we have a uniform lower bound for the capacity,
and so we conclude as in Case~1).

Hence we suppose that~$Q_{\frac{3}{4}R}$ is
not contained in~$K$, and we take a small square centered at~$q$,
say~$Q_{\frac{R}{8}}(q)$, such
that~$Q_{\frac{R}{8}}(q)\subset Q_{\frac{3}{4}R}$
(see Figure~1(B)).

Now, recalling~\eqref{positivity}, we have that we can deal with
a one-phase problem in the square~$Q_{\frac{R}{8}}(q)$.
Hence, from Theorem~4.4 in~\cite{DP2} we obtain that
$$ |\{u_j^+=0\}\cap Q_{\frac{3}{8}R}(q)|\ge c_0 R^2,$$
for some universal constant~$c_0>0$, see Figure~1(B).
Again this implies a lower bound for the capacity, thanks to \eqref{lower-cap-L},
and so we conclude as in Case~1).

Suppose now that there is another component $K_2\subset\{u^+_j=0\}$
such that~$K_2\cap Q_{\frac{3}{2}R}\neq\emptyset$ (that is~$u_j$ may change sign). 
Then, as before, either $\partial K_2\cap\partial Q_{2R}\neq\emptyset$ 
or $\partial K\Subset Q_{2R}$. 
In the first case, we obtain a lower bound for the capacity
reasoning as in Case $2b_1)$. In the second case, we use again the maximum principle 
to reduce ourselves to a one-phase minimization problem 
and, from the density estimate for the zero set, we get a lower bound for the capacity. 

\medskip

{\bf Case~3):}  Finally we deal with the last case,
which is the easiest one. In fact the
proof follows as in {\bf Case 2a)} if we replace there $Q_{\frac{3R}2}$ with $Q_{2R}$.

Thus, since $p_j>2$, for any~$p_0/2<\ell<2$ we obtain
the claim in~\eqref{bb55} also for squares that do not touch~$\partial\{u_j^+=0\}$.

Combining all the cases treated above, we can see that for any square
$Q_{2R}\subset B_3$ and some fixed $\ell$ with~$p_0/2<\ell<2$
there exists a tame constant $C>0$ such that there holds
$$ \left(\fint_{Q_R}|\nabla u^+_j|^{2}\right)^{1/2}\le
C\,\left(\fint_{Q_{2R}}|\nabla u^+_j|^{\ell}\right)^{1/\ell}.$$
Therefore we can apply the Gehring's Lemma (see Proposition \ref{prop-Geh}, and for instance \cite{Giaq} for the proof) and we get
that there exists~$q>2$ such that
$$ \| \nabla u^+_j\|_{L^q(Q_R)}\le C, $$
for a suitable~$C>0$.
By a covering argument, this implies the desired result.
\end{proof}
%\end{comment}

From the uniform estimates in $W^{1,q}_{loc}(B_3)$, with $q>2$,
and the Sobolev's embedding Theorem we immediately get the following:

\begin{corollary}\label{cor-contn}
The functions $u_j^\pm$ are uniformly continuous in $B_2$.
\end{corollary}

\medskip

\subsection{Step 3: Linearity in $B_3\setminus B_1$}\label{sec:step3}
Thanks to Lemma~\ref{lem-Gehring} and a standard compactness argument,
we conclude that
\begin{equation}\label{starstar}
{\mbox{ $\nabla u^{\pm}_j$ converges strongly in~$L^{q'}(B_1)$,
for any~$q'<q$, with $q>2$, to some~$\nabla u_0^{\pm}$.}}
\end{equation}
Moreover, Lemma~\ref{lem-nondeg} implies that
both~$u^+_0$ and~$u^-_j$ are non-degenerate.
Therefore, since~$p_j\rightarrow 2$ as~$j\to+\infty$,
from~\eqref{pr-probe} we deduce that
\begin{eqnarray}\label{bb90}
\liminf_{j\to \infty} \varphi_{p_j}(1,u_j,0) &\ge & \liminf_{j\to \infty}
\varphi_{p_j}(3,u_j,0)\\\nonumber
&=& \liminf_{j\to \infty}3^{-4}\int_{B_3}|\nabla u^{+}_j|^{p_j}\,\int_{B_3}|\nabla u^{-}_j|^{p_j} \\\nonumber
&\ge &3^{-4}\int_{B_3}|\nabla u^{+}_0|^{2}\,\int_{B_3}|\nabla u^{-}_0|^{2}
\equiv \varphi_2(3, u_0, 0),
\end{eqnarray}
where the last line follows from the semicontinuity of the Dirichlet's integral.

On the other hand, \eqref{starstar} implies strong convergence
of the gradient in~$L^2(B_1)$, since~$q>2$ in Lemma~\ref{lem-Gehring}.
Therefore
$$ \liminf_{j\to \infty} \varphi_{p_j}(1,u_j,0) =\liminf_{j\rightarrow+\infty}\int_{B_1}|\nabla u^+_j|^{p_j}\,\int_{B_1}|\nabla u^-_j|^{p_j}=
\int_{B_1}|\nabla u^+_0|^2\,\int_{B_1}|\nabla u^-_0|^2 .$$
Hence,
\begin{equation}\label{bb91}
\varphi_2(3,u_0,0)\le  \int_{B_1}|\nabla u^+_0|^2\,\int_{B_1}|\nabla u^-_0|^2 =\varphi_2(1, u_0, 0).
\end{equation}

Now, we observe that~$u_0^+$ and $u_0^-$
are non-negative subharmonic functions
with disjoint supports fulfilling the conditions
of Theorem~\ref{Luis-mon},
and so the monotonicity of~$\varphi_2$ implies that
$$ \varphi_2(1,u_0,0)\le\varphi_2(3,u_0,0).$$
This and~\eqref{bb91} give that~$\varphi_2$ is constant in~$B_3\setminus B_1$.
Thus, Theorem \ref{Luis-mon} yields that
$u^+_0$ and $u^-_0$ must be linear in $B_3\setminus B_1$, say, 
$u_0^+=\alpha x_1^+ $ and $u_0^-=\beta x_1^-$, for some $\alpha$ and $\beta>0$.

\medskip

%\section{Case $S^\pm_j$ unbounded: Anisotropic scaling}

\subsection{Step 4: Filling in the gap}\label{sec:anisotropic}

In this subsection we want to show that~$u^+_0$
and~$u^-_0$ are linear in~$B_3$, and this will give a contradiction
with~\eqref{non-flat-0}. For this, we will prove that either $u_0^+$ in $\{u_0>0\}$ or $u_0^-$ in $\{u_0<0\}$ is harmonic, in order
to employ some unique continuation result.

Let us show that
\begin{equation}\label{pharm}
{\mbox{$u_0^+$ is harmonic in~$\{u_0>0\}$}}\end{equation}
(the proof for~$u_0^-$ is analogous).
We take a point~$x_0\in\Omega$ such that~$u_0(x_0)>0$,
then, thanks to the uniform convergence of~$u_j$ to~$u_0$,
we have that~$u_j(x_0)>0$ for~$j$ large enough. Therefore,
Corollary~\ref{cor-contn}
implies that there exists a small~$\delta=\delta(x_0)>0$ 
such that~$u_j>0$ in~$B_\delta(x_0)$, and so we can use P.1 in Proposition~\ref{prop:tec} 
to obtain that
$$ \Delta_{p_j}u_j=0 \quad {\mbox{ in }} B_\delta(x_0).$$
Therefore, for any~$\psi\in C^\infty_0(B_\delta(x_0))$, we have that
$$ \int_{B_\delta(x_0)}|\nabla u_j|^{p_j}\le
\int_{B_\delta(x_0)}|\nabla u_j+\nabla\psi|^{p_j}.$$
Taking the limit as~$j\to+\infty$ we have that
\begin{equation}\label{pharm-1}
\int_{B_\delta(x_0)}|\nabla u_0|^{2}\le
\int_{B_\delta(x_0)}|\nabla u_0+\nabla\psi|^{2}, \quad {\mbox{ for any }}
\psi\in C^\infty_0(B_\delta(x_0))
\end{equation}
(recall that~$\phi$ is fixed and that we have strong convergence
of~$\nabla u_j$ to~$\nabla u_0$ in~$L^2_{loc}(B_3)$).
By a density argument, from~\eqref{pharm-1} we get
$$ \int_{B_\delta(x_0)}|\nabla u_0|^{2}\le
\int_{B_\delta(x_0)}|\nabla v|^{2}, \quad {\mbox{ for any }}
v\in W^{1,2}(B_\delta(x_0)) \ {\mbox{ s.t. }} v-u_0\in W^{1,2}(B_\delta(x_0)).$$
Thus we conclude that
$$ \Delta u_0=0  \quad {\mbox{ in }} B_\delta(x_0).$$
Since~$u_0$ is a continuous function, this implies~\eqref{pharm}.

From Step~3 and~\eqref{pharm}, and applying the Unique Continuation Theorem (see~\cite{Lin}), 
we obtain that
\begin{equation}\label{aldjk}
{\mbox{$u^+_0$ and~$u^-_0$ are linear in~$B_3$.}}\end{equation}
On the other hand, the uniform convergence of~$u_j$ to~$u_0$, as $j\to+\infty$, 
implies that~\eqref{non-flat-0} holds true,
and so the level sets of~$u_0$ are not flat in~$B_1$. 
Indeed, by the uniform convergence, for any $\epsilon >0$ there is $j_0$ such that $|Cx_1-u_j^+(x)|<\epsilon$
whenever $j>j_0$, where we assume that $u_0^+(x)=Cx_1$ for some 
constant $C>0$.  Since $\fb {u_j}$ is $h_0$ thick in $B_1$ it follows that 
there is $y_j\in \fb {u_j}\cap B_1$ such that $y_j=e_1h_0/2+t_je_2$, for some $t_j\in \R$, where 
$e_1$ is the unit direction of the $x_1$-axis and $e_2\perp e_1$. Then we have that 
$|C\frac{h_0}2-0|=|u_0^+(y_j)-u_j^+(y_j)|<\epsilon$ which is %a contradiction if $\epsilon$ is small.
%%
%This is 
in contradiction
with~\eqref{aldjk}, and thus concludes the proof
of Theorem~\ref{TH2}.

\medskip

\section{Proof of Theorem \ref{TH2-prime}}\label{sec-th2-0-pr}

In this section we prove Theorem \ref{TH2-prime}. For this, we recall Corollary \ref{coro3.4} and 
we square \eqref{cod-ball}: we have 
\begin{equation}
\left(\frac1r\fint_{B_r(x_0)}u^+\right)^2+ \left(\frac1r\fint_{B_r(x_0)}u^-
\right)^2\le C^2+\frac2{r^2}\fint_{B_r(x_0)}u^+
\fint_{B_r(x_0)}u^-,
\end{equation}
where $C>0$ is the constant appearing in Corollary \ref{coro3.4}. 

Now we set $u_{r}^\pm(x):=u^\pm(rx)$. So from the H\"older inequality, 
the Poincar\'e inequality \eqref{Poinc} 
and \eqref{lower-cap-H}, we have that, for any $1<\ell<2$,
\begin{eqnarray*}
&&\fint_{B_r(x_0)}u^\pm(x)\,dx \leq \left(\fint_{B_r(x_0)}(u^\pm)^\ell(x)\,dx\right)^{\frac1\ell}
=\left(\fint_{B_1(x_0)}(u^\pm_r)^{\ell}(y)\,dy \right)^{\frac1\ell}\\
&&\qquad \leq C_1 \left(\fint_{B_1(x_0)}|\nabla u^\pm_r(y)|^{\ell}\,dy \right)^{\frac1\ell}
=  
C_1\left(r^\ell\fint_{B_r(x_0)}|\na u^\pm|^\ell\right)^{\frac1\ell}\\&&\qquad =
C_1\,r\left(\fint_{B_r(x_0)}|\na u^\pm|^\ell\right)^{\frac1\ell},\end{eqnarray*}
for some $C_1>0$. However, from 
H\"older's inequality we have for $p>2>\ell>1$
$$\left(\fint_{B_r(x_0)}|\na u^\pm|^\ell\right)^{\frac1\ell}\leq \left(\fint_{B_r(x_0)}|\na u^\pm|^p\right)^{\frac1p}.$$Therefore 
$$\left(\frac1r\fint_{B_r(x_0)}u^+\right)^2+ \left(\frac1r\fint_{B_r(x_0)}u^-\right)^2\le 
C^2+ C_2(\phi_p(r, u, x_0))^\frac1p,$$
for some $C_2>0$. 

Let now $r_k:=3^{-k_0-k}$, for any $k\in\mathbb N$, where $k_0$ is the 
smallest positive integer such that $3^{-k_0}<r_0$. If $3^{-m-1}\le r\le 3^{-m}$, 
for some $m\in \mathbb N$, then 
$$\phi_p(r, u, x_0)\le C_3 \phi_p(3^{-m}, u, x_0)\le C_3 \phi_p(3^{-k_0}, u, x_0)$$ 
for some $C_3>0$, 
implying that 
$$\left(\frac1r\fint_{B_r(x_0)}u^\pm\right)^2\le C^2+ C(\phi_p(3^{-k_0}, u, x_0))^\frac1p,$$
for suitable $C_4>0$. 
Hence, P.4 in Proposition \ref{prop:tec} and 
the weak maximum principle (see Corollary 3.10 in \cite{MZ}) 
imply the estimate $\sup_{B_r(x_0)}|u|\leq Cr$. This completes the proof of Theorem \ref{TH2-prime}.

%%%%%%%%%%%%%%%%%%%%%%%%%%%
%				
%			 	SECTION													
%																						
%%%%%%%%%%%%%%%%%%%%%%%%%%%
\section{Viscosity solutions}\label{sec:visc}

In order to apply  the regularity theory for free boundary problems developed 
for the viscosity solutions in~\cite{LN1,LN2}
we shall observe that any weak $W^{1,p}$ minimizer is also viscosity solution
(see Definition~2.4 in~\cite{Luis} for the case $p=2$). For this, we denote
by~$\Omega^+(u):=\{u>0\}$ and~$\Omega^-(u):=\{u<0\}$.
Moreover,
$$G(u^+_\nu,u^-_\nu):=(u^+_\nu)^p-(u^-_\nu)^p-\frac\Lambda {p-1}$$
is the flux balance across the free boundary,
where~$u^+_\nu$ and~$u^-_\nu$ are the normal derivatives in the inward direction
to~$\Omega^+(u)$ and~$\Omega^-(u)$, respectively (recall that~$\Lambda=\lambda^p_+-\lambda^p_-$). 

We recall the definition of viscosity solutions for the case $p\neq 2$ (see Definition 4.1 in \cite{DK}).

\begin{defn}\label{def:visc}
Let~$\Omega$ a bounded domain in~$\R^2$ and let~$u$ be a continuous function in~$\Omega$. We say that~$u$ is a viscosity solution
in~$\Omega$ if
\begin{itemize}
\item[i)] $\Delta_p u=0$ in~$\Omega^+(u)$ and~$\Omega^-(u)$,
\item[ii)] along the free boundary~$\Gamma=\partial\{u>0\}\cup\partial\{u<0\}$, $u$ satisfies the free boundary condition,
in the sense that:
\begin{itemize}
\item[a)] if at~$x_0\in\Gamma$ there exists a ball~$B\subset\Omega^+(u)$
such that~$B\cap\Gamma=\{x_0\}$ and
\begin{equation}\label{visc1}
u^+(x)\ge\alpha\langle x-x_0,\nu\rangle^+ + o(|x-x_0|), \ {\mbox{ for }} x\in B,
\end{equation}
\begin{equation}\label{visc2}
u^-(x)\le\beta\langle x-x_0,\nu\rangle^- + o(|x-x_0|), \ {\mbox{ for }} x\in B^c,
\end{equation}
for some $\alpha>0$ and~$\beta\ge0$, with equality along every non-tangential domain,
then the free boundary condition is satisfied
$$ G(\alpha,\beta)\le0, $$
\item[b)] if at~$x_0\in\Gamma$ there exists a ball~$B\subset\Omega^-(u)$
such that~$B\cap\Gamma=\{x_0\}$ and
$$ u^-(x)\ge\beta\langle x-x_0,\nu\rangle^- + o(|x-x_0|), \ {\mbox{ for }} x\in B, $$
$$ u^+(x)\le\alpha\langle x-x_0,\nu\rangle^+ + o(|x-x_0|), \ {\mbox{ for }} x\in\partial B, $$
for some $\alpha\ge0$ and~$\beta>0$, with equality along every non-tangential domain,
then
$$ G(\alpha,\beta)\ge0. $$
\end{itemize}
\end{itemize}
\end{defn}

With this notion of viscosity solutions, in \cite{DK} we prove the following:

\begin{theorem}\label{TH:viscosity}
Let~$u\in W^{1,p}(\Omega)$ be a minimizer of~\eqref{Ju}.
Then, $u$ is also a viscosity solution in the sense of Definition~\ref{def:visc}.
\end{theorem}

See Theorem 4.2 in \cite{DK} for the proof of Theorem~\ref{TH:viscosity}.

We also recall the notion of $\e-$monotonicity of a viscosity solution to our free boundary problem.
\begin{defn}
We say that
$u$ is $\epsilon-$monotone
if there are a unit vector $e$ and an angle $\theta_0$ with 
$\theta_0 > \frac\pi 4$ (say) and $\epsilon >0$ (small)
such that, for every $\epsilon'\ge \epsilon $,
\begin{equation}\label{e-mon}
\sup_{B_{\epsilon' \sin\theta_0} (x)} u(y -\epsilon ' e) \le u(x).
\end{equation}
\end{defn}

We define~$\Gamma(\theta_0,e)$ the cone with axis~$e$ and opening~$\theta_0$. 

\begin{defn}
We say that $u$ is $\epsilon-$monotone in the 
cone~$\Gamma(\theta_0,\e)$ if it is~$\epsilon-$monotone 
in any direction~$\tau\in\Gamma(\theta_0,\epsilon)$. 
\end{defn}

One can interpret the $\e-$monotonicity of $u$ as closeness of 
the free boundary to a Lipschitz graph with Lipschitz constant sufficiently close to 
$1$ if we depart from  the free boundary in directions $e$ at distance $\e$ and higher. 
The exact value of the Lipschitz constant is given 
by~$\left(\tan\frac{\theta_0}{2}\right)^{-1}$. 
Then the ellipticity propagates to 
the free boundary via Harnack's inequality giving that $\Gamma$ is Lipschitz. 
Furthermore, Lipschitz free boundaries are, in fact, $C^{1, \alpha}$ regular.

\medskip 

For $p=2$ this theory was founded by L. Caffarelli, see~\cite{Caffa1, Caffa3, Caffa2}. 
Recently J. Lewis and K. Nystr\"om proved that this  theory is valid for all $p>1$, see \cite{LN1, LN2}.

\medskip

For viscosity solutions we replace the $\epsilon-$monotonicity 
with slab flatness measuring the thickness of $\fb u\cap B_r(x)$ in terms of the quantity 
$h(x,r)$ introduced in \eqref{min-h}. In other words, $h(x, r)$ measures
how close the free boundary is to a pair of parallel planes in a ball 
$B_r(x)$ with $x\in \Gamma.$ Clearly, planes are 
Lipschitz graphs in the direction of the normal, therefore 
the slab flatness of $\Gamma$ is a particular case of $\e-$monotonicity of $u$.
  
Hence, under $h_0-$flatness of the free boundary we can 
reformulate the regularity theory ``flatness implies $C^{1, \alpha}$'' as follows:
\begin{theorem}\label{TH1}
Let $x_0\in \fb{u}$ and $r>0$ such that~$B_r(x_0)\subset\Omega$. 
Then there exists $h>0$ such that if
$\Gamma\cap B_r(x_0)\subset \{x\in \R^N : -hr<(x-x_0)\cdot\nu<hr\}$ 
then $\Gamma\cap B_{r/2}(x_0)$ is locally~$C^{1,\alpha}$ 
in the direction of $\nu$, for some~$\alpha\in(0,1)$.
\end{theorem}

%%%%%%%%%%%%%%%%%%%%%%%%%%%
%				
%			 	SECTION													
%																						
%%%%%%%%%%%%%%%%%%%%%%%%%%%
\section{Geometry of eigenvalues}\label{sec-eigen-1}

Here we present some results that are related to the characteristic numbers 
and the eigenvalues of the $p$-Laplace-Beltrami operator for $p\neq 2$. 

\subsection{Homogeneous $p$-harmonic functions in complementary cones}\label{sec-eigen}

Let us consider 
$$\varphi_p(R,u_1,u_20):=\frac{1}{R^4}\int_{B_R}|\nabla u_1|^p\,\int_{B_R}|\nabla u_2|^p,
$$
for given $u_i=r^{\lambda_i}g_i(\theta)$, with $i=1, 2$ such that $u_1, u_2$ are $p-$harmonic 
in two complementary cones. Here $r, \theta$ are the polar coordinates.
We will show an estimate on the eigenvalues $\lambda_1$ and $\lambda_2$ of the $p$-Laplace-Beltrami 
operator, namely we prove that
\begin{equation}\begin{split}\label{label71}
&\sqrt{\lambda_1(\lambda_1(p-1)+2-p)}+\sqrt{\lambda_2(\lambda_2(p-1)+2-p)}\geq 2,\\
&{\mbox{with equality if and only if both functions are linear.}}\end{split}\end{equation}
In turn, this implies that $\phi_p(R, u_1, u_2 0)$ 
is non-decreasing in $R$. Furthermore, $\phi_p(R, u_1, u_2 0)$
is constant if and only if $\lambda_1=\lambda_2=0$.

\subsection{Properties of Eigenvalues}
In this section we prove a relation between the eigenvalues of the $p-$Laplace-Beltrami operator
that correspond to two complementary cones. We begin with an existence 
result of P. Tolksdorf \cite{T1}, page 780, Theorem 2.1.1 and Corollary 2.1.

\begin{theorem}\label{Tolksdorf-eigenvalue}
 Let $S:=(0,\omega)$, with $\omega\in [0, 2\pi]$. 
 Then there exists a solution $(\lambda, \phi(\theta))$, with $\theta \in S$ of
\begin{equation}\label{starewrwe}
\left\{
\begin{array}{ccc}
 -\frac{d}{d\theta} \left\{(\lambda^2\phi^2+\phi_\theta^2)^{\frac{p-2}{2}} \phi_\theta \right\} =\lambda(\lambda(p-1)+2-p)(\lambda^2\phi^2+\phi_\theta^2)^{\frac{p-2}{2}} \phi  \  {\rm in}\  S,\\
\phi(\theta)=0\  {\rm on}\  \partial S,
\end{array}
\right.
\end{equation}
such that
\begin{eqnarray*}
&&\lambda>\max\left\{0, \frac{p-2}{p-1}\right\},\  \phi>0 \ {\rm in } \ S,\\\nonumber
{\mbox{and }} \; &&\phi^2+\phi_\theta^2>0 \ {\rm in } \ S.
\end{eqnarray*}
Furthermore any two solutions are constant multiples of each other.
\end{theorem}

M. Dobrowolski computed explicitly the value of $\lambda$ in \eqref{starewrwe}, 
see \cite{Dob} page 187, Theorem 1:

\begin{theorem}
Let $\phi$ be given by Theorem \ref{Tolksdorf-eigenvalue}. Then
\begin{eqnarray}\label{lambda}
 \lambda=\left\{
\begin{array}{ccc}
 s+\sqrt{s^2+\frac 1\rho}, &\omega \le \pi, \\
s-\sqrt{s^2+\frac 1\rho}, &\pi\le \omega< 2\pi,\\
\frac{p-1}{p}, &\omega=2\pi,
\end{array}
\right.
\end{eqnarray}
where
\begin{eqnarray}
\rho &:=&\left(\frac\omega\pi-1\right)^2-1 \label{rho},\\
{\mbox{and }}\; s&:=&\frac{(\rho-1)p-2\rho}{2\rho(p-1)}=\frac{p-2}{2(p-1)}+\frac{p}{2(p-1)}\left[-\frac 1\rho\right]\label{s}.
\end{eqnarray}
\end{theorem}

The key result of this section is contained in the following lemma.

\begin{lemma}\label{lemma123}
Let $\lambda_1$ be the solution of \eqref{starewrwe}  
for $S_1:=(0, \omega)$ and $\lambda_2$ for the complementary arc $S_2:=(\omega, 2\pi)$. Then
\begin{equation}\label{sum-geq-2}
 \sqrt{\lambda_1(\lambda_1(p-1)+2-p)}+\sqrt{\lambda_2(\lambda_2(p-1)+2-p)}\geq 2.
\end{equation}
Furthermore equality holds if and only if $\lambda_1=\lambda_2=1$, i.e. for half circles $S=(0, \pi)$.
\end{lemma}

\begin{proof}
Without loss of generality we may assume that $\omega\le \pi$. Next let us notice that the eigenvalue $\lambda$ is determined by the
size of the arc only. Hence for $S_2$ we have  by (\ref{rho})	
\begin{equation*}
 \rho_2=\left(\frac{2\pi -\omega}{\pi}-1\right)^2-1=\left(1-\frac\omega\pi\right)^2-1.
\end{equation*}
Thus $\rho_1=\rho_2=\rho$ and from (\ref{s}) we infer that $s_1=s_2=s$.
In order to prove (\ref{sum-geq-2}) it is enough to check that
%\arrayseplength{2pt}
\begin{equation}\label{aggdopo}\begin{split}
&I+2\sqrt{II}\ge 4,\\
{\mbox{where }} \; &I:=\lambda_1(\lambda_1(p-1)+2-p)+\lambda_2(\lambda_2(p-1)+2-p) \\
{\mbox{and }}\; &II:=\lambda_1\lambda_2(\lambda_1(p-1)+2-p)(\lambda_2(p-1)+2-p).
\end{split}\end{equation}
In order to prove this, we notice that, by \eqref{lambda}, 
\begin{eqnarray*}
\left\{
\begin{array}{ccc}
 \lambda_1+\lambda_2&=&2s\\
\lambda_1\lambda_2&=&-\frac1\rho\\
\lambda_1^2+\lambda_2^2&=&4s^2+\frac 2\rho
\end{array}
\right.
\end{eqnarray*}
which gives
\begin{eqnarray}\label{I1}
 I=(\lambda_1^2+\lambda_2^2)(p-1)+(2-p)(\lambda_1+\lambda_2)= (p-1)\left(4s^2+\frac2\rho\right)+2s(2-p).
\end{eqnarray}

For convenience we introduce a new quantity
\begin{equation}\label{t-def}
t:=-\frac1\rho
\end{equation} 
and notice that, by \eqref{rho}, we have that $t\geq 1$ and, by \eqref{s}, one has
\begin{equation}\label{st}
 s=\frac{1}{2(p-1)}(p-2+pt).
\end{equation}
Hence \eqref{I1} can be manipulated further in the following way:
\begin{eqnarray*}
I&=&2\left[(p-1)\left(2s^2+\frac1\rho\right)+s(2-p)\right]\\
&=&2\left[s\left\{2s(p-1)+(2-p)\right\}+ (p-1)\frac1\rho\right]\\
&=&2\left[s\left\{(p-2+pt)+(2-p)\right\}+ (p-1)\frac1\rho\right]\\
&=&2 \left[spt-t(p-1)\right]\\
&=&2t[sp-(p-1)].\\
%&=&2t[]
\end{eqnarray*}

Similarly, using \eqref{lambda} and \eqref{st} we get
\begin{eqnarray*}
 II&=&\lambda_1\lambda_2(\lambda_1(p-1)+2-p)(\lambda_2(p-1)+2-p)\\
&=&-\frac1\rho\left[\lambda_1 \lambda_2 (p-1)^2+(p-1)(2-p)(\lambda_1+\lambda_2)+(2-p)^2\right]\\
&=&-\frac1\rho\left[-\frac1\rho (p-1)^2+2s(p-1)(2-p)+(2-p)^2\right]\\
&=&t\left[t(p-1)^2+2s(p-1)(2-p)+(p-2)^2\right]\\
&=&t\left[t(p-1)^2+(p-2+pt)(2-p)+(2-p)^2\right]\\
&=&t\left[t(p-1)^2+pt(2-p)\right]\\
&=&t^2.
\end{eqnarray*}
Thus, putting together the last two formulas,
\begin{equation}\begin{split}\label{sharq-0}
I+2\sqrt{II}=\,& 2t[sp-(p-1)]+2t\\
=\,&2t\left[sp-(p-1)+1\right]\\
=\,&2t\left[\frac{p}{2(p-1)}(p-2+pt)-(p-1)+1\right]\\
=\,&\frac{t}{p-1}\left[p(p-2+pt)-2(p-1)(p-2)\right]\\
=\,&\frac{t}{p-1}\left[p^2-2p+p^2t-2p^2+6p-4)\right]\\
=\,&\frac{t}{p-1}\left[p^2(t-1)+4(p-1)\right]\\
=\,&4t+\frac{p^2}{p-1}t(t-1)\\
\geq\,& 4
\end{split}\end{equation}
since $t\geq 1$, which implies \eqref{aggdopo} and finishes the proof of Lemma \ref{aggdopo}.
\end{proof}

%%%%
%%%% Logarithmic derivative 
%%%%

\subsection{Computing the logarithmic derivative}
In what follows, fixed $u_1$ and $u_2$, 
we put $\phi(R):=\phi_p(R, u_1,u_2,0)$. 
In order to prove that $\phi$ is non-decreasing in $R$, 
it is enough to prove that $\phi'(R)\ge 0$ for $R=1$, 
since $\phi$ is scale-invariant. 
For this, let $J_i:=\int_{B_R}|\na u_i|^p$. Then we have 
\begin{eqnarray}
\displaystyle (\log \phi(R))'=\frac{\phi'(R)}{\phi(R)}&=&
\frac{J'_1(R)}{J_1(R)}+\frac{J_2'(R)}{J_2(R)}-\frac4R\\\nonumber
&=&\displaystyle \frac{\int_{\partial B_R}|\na u_1|^p}{\int_{B_R}|\na u_1|^p}+\frac{\int_{\partial B_R}|\na u_2|^p}{\int_{B_R}|\na u_2|^p}-\frac4R\\\nonumber
&=&\displaystyle \frac1R\left(\frac{R\int_{\partial B_R}|\na u_1|^p}{\int_{B_R}|\na u_1|^p}+\frac{R\int_{\partial B_R}|\na u_2|^p}{\int_{B_R}|\na u_2|^p}-4\right).\\\nonumber
\end{eqnarray}

Next we notice that
\begin{eqnarray}\label{sharq-1}
 \int_{B_1}|\na u_i|^p&=&\int_{\partial B_1}|\na u_i|^{p-2}u_i\frac{\partial u_i}{\partial \nu}\\\nonumber
&\leq&\left[\int_{\partial B_1}|\na u_i|^{p-2}u^2_i\int_{\partial B_1}|\na u_i|^{p-2}
u_{i,{\rm rad}}^2\right]^{\frac12}, \quad i=1,2,
\end{eqnarray}
where $u_{i,\nu}=u_{i,{\rm rad}}$ is the radial derivative (in direction of the outer unit normal $\nu$ of unit circle).

Next decomposing $|\na u_i|^2$ into the sum of the squares of the radial and tangential derivative, $u_{i,\theta}$, we obtain
\begin{eqnarray}\label{sharq-2}
 \int_{\partial B_1}|\na u_i|^p&=&\int_{\partial B_1}|\nabla u_i|^{p-2}(u_{i,{\rm rad}}^2+u_{i, \theta}^2)\geq\\\nonumber
&\geq& 2\left[\int_{\partial B_1}|\na u_i|^{p-2}u_{i,{\rm rad}}^2\int_{\partial B_1}|\nabla u_i|^{p-2}u_{i,\theta}^2\right]^{\frac12}.
\end{eqnarray}

Hence to prove that $\phi$ is monotone it is enough to check that
\begin{eqnarray*}
 \underbrace{\left[\frac{\displaystyle \int_{S_1}|\nabla u_1|^{p-2}u_{1,\theta}^2}{\displaystyle\int_{ S_1}|\na u_1|^{p-2}u_1^2}\right]^{\frac12}}_{H(u_1)} +
\underbrace{\left[\frac{\displaystyle\int_{S_2}|\nabla u_2|^{p-2}u_{2,\theta}^2}{\displaystyle\int_{ S_2}|\na u_2|^{p-2}u_2^2}\right]^{\frac12}}_{H(u_2)}-2\geq 0.
\end{eqnarray*}
Here $S_i:=\supp u_i\cap\partial B_1$. 
For the solutions $(\lambda_i, \phi_i)$
of the eigenvalue problem on $S_i$ stated in Theorem \ref{Tolksdorf-eigenvalue}
and \eqref{sum-geq-2} we infer that $H(u_1)+H(u_2)\ge 2$, thanks to \eqref{sum-geq-2}. 

Recalling the notation in \eqref{aggdopo} and 
observing that in \eqref{sharq-0} the equality $I+2\sqrt{II}=4$ holds if and only if 
$(t-1)(4+\frac{tp^2}{p-1})=0$, we conclude that $t=1$.
On the other hand, the equality holds in \eqref{sum-geq-2} if and only if $t=1$, i.e. 
by \eqref{t-def} when $\rho=-1$, and hence, 
in view of \eqref{rho}, when $\omega=\pi$, which corresponds to the half circle. 
This implies that $\phi$ is non-decreasing, 
and in turn shows \eqref{label71}.

\end{document}